\documentclass[10pt]{article}

\usepackage[
  letterpaper,
  textwidth=5.125in,
  textheight=8.25in
]{geometry}

\usepackage{amsmath,amssymb,amsfonts,amsthm,amsopn}
\usepackage{mathtools}

\usepackage{graphicx}
\usepackage{epstopdf}
\usepackage{booktabs}
\usepackage{enumitem}
\usepackage{float}
\usepackage{subcaption}
\usepackage{cancel}
\usepackage{etoolbox}
\usepackage{algorithmicx}
\usepackage{algpseudocodex}

\usepackage{hyperref}
\hypersetup{
  colorlinks=true,
  citecolor=blue,
  linkcolor=blue,
  urlcolor=blue,
  pdftitle={Mixed precision Newton's method for optimization},
  pdfauthor={N. Brisebarre, G. Carrino, T. Mary, and E. Riccietti}
}

\usepackage{zref-clever}

\ifpdf
  \DeclareGraphicsExtensions{.eps,.pdf,.png,.jpg}
\else
  \DeclareGraphicsExtensions{.eps}
\fi

\newtheorem{theorem}{Theorem}[section]
\newtheorem{lemma}[theorem]{Lemma}

\newtheorem{errormodel}[theorem]{Error Model}

\theoremstyle{definition}
\newtheorem{definition}[theorem]{Definition}

\theoremstyle{remark}

\title{Mixed precision Newton's method for optimization\thanks{Version of \today.}}

\author{
Nicolas Brisebarre\thanks{CNRS, ENS de Lyon, Inria, Université Claude Bernard Lyon 1, LIP, UMR 5668, Lyon, France (nicolas.brisebarre@cnrs.fr)}
\and
Giuseppe Carrino\thanks{ENS de Lyon, CNRS, Inria, Université Claude Bernard Lyon 1, LIP, UMR 5668, Lyon, France (giuseppe.carrino@ens-lyon.fr)}
\and
Theo Mary\thanks{Sorbonne Université, CNRS, LIP6, F-75005 Paris, France (theo.mary@lip6.fr)}
\and
Elisa Riccietti\thanks{ENS de Lyon, CNRS, Inria, Université Claude Bernard Lyon 1, LIP, UMR 5668, Lyon, France (elisa.riccietti@ens-lyon.fr)}
}

\date{\today}

\zcsetup{abbrev=false}
\zcsetup{nameinlink=true}
\AddToHook{env/lemma/begin}{\zcsetup{countertype={theorem=lemma}}}
\AddToHook{env/corollary/begin}{\zcsetup{countertype={theorem=corollary}}}
\AddToHook{env/proposition/begin}{\zcsetup{countertype={theorem=proposition}}}
\AddToHook{env/definition/begin}{\zcsetup{countertype={theorem=definition}}}
\AddToHook{env/remark/begin}{\zcsetup{countertype={theorem=remark}}}
\AddToHook{env/errormodel/begin}{\zcsetup{countertype={theorem=errormodel}}}
\AddToHook{cmd/appendix/before}{\zcsetup{countertype={section=appendix}}}
\AddToHook{env/assumption/begin}{%
   \zcsetup{countertype={theorem=assumption}}}

\zcRefTypeSetup{equation}{
  name-sg=,
  name-pl=,
  Name-sg=,
  Name-pl=,
  name-sg-ab=,
  name-pl-ab=,
  Name-sg-ab=,
  Name-pl-ab=,
}

\zcRefTypeSetup{errormodel}{
    cap=true,
    Name-sg=Error Model,
    Name-pl=Error Models,
    Name-sg-ab=Error Model,
    Name-pl-ab=Error Models,
}

\zcRefTypeSetup{assumption}{cap=true}
\zcRefTypeSetup{remark}{cap=true}
\zcRefTypeSetup{theorem}{cap=true}
\zcRefTypeSetup{lemma}{cap=true}
\zcRefTypeSetup{corollary}{cap=true}
\zcRefTypeSetup{proposition}{cap=true}
\zcRefTypeSetup{algorithm}{cap=true}
\zcRefTypeSetup{appendix}{cap=true}
\zcRefTypeSetup{definition}{cap=true}
\zcRefTypeSetup{figure}{cap=true}

\setlist[enumerate]{topsep=1mm}
\setlist[itemize]{topsep=1mm}

\numberwithin{equation}{section}
\numberwithin{figure}{section}
\numberwithin{table}{section}

\newcommand{\xh}{\widehat{x}}
\newcommand{\dhat}{\widehat{d}}

\newcommand{\bbR}{\mathbb{R}}

\newcommand{\Hi}{H(\xh_i)}

\newcommand{\Hstar}{H(x^*)}
\newcommand{\gi}{g(\xh_i)}
\newcommand{\ginext}{g(\xh_{i+1})}

\newcommand{\Ji}{J(\xh_i)}

\newcommand{\Jstar}{J(x^*)}
\newcommand{\Si}{S(\xh_i)}

\newcommand{\Sstar}{S(x^*)}
\newcommand{\eh}{\epsilon^H_i}
\newcommand{\eg}{\epsilon^g_i}

\newcommand{\ei}{\epsilon_i}

\newcommand{\ki}{\zeta_i}
\newcommand{\errg}{e_i^g}
\newcommand{\errh}{E_i^H}
\newcommand{\errsum}{e_i^+}

\newcommand{\errconvrate}{\theta_i}

\DeclarePairedDelimiter{\abs}{\lvert}{\rvert}
\DeclarePairedDelimiter{\norm}{\lVert}{\rVert}

\newcommand{\limacc}{\gamma_i}
\newcommand{\limg}{\psi_i}

\floatstyle{ruled}
\newfloat{algorithm}{tbp}{loa}
\floatname{algorithm}{Algorithm}

\apptocmd{\sloppy}{\hbadness 10000\relax}{}{}

\DeclareMathOperator{\fl}{\operatorname{f\kern.2ptl}}
\def\op{\mathbin{\mathrm{op}}}

\begin{document}

\maketitle

\begin{abstract}
Second-order optimization methods, such as Newton’s algorithm, achieve fast local
convergence and high accuracy, but their practical use is often limited by high
computational costs. To mitigate this issue, variants such as inexact and
quasi-Newton methods are widely used. A complementary and promising approach to
improve the efficiency of the method is to employ mixed precision arithmetic,
using different floating-point precisions for different operations, based on
their impact on the convergence and accuracy of the method. In this work, we
perform an error analysis of Newton's method accounting for different sources of
inexactness, including approximations and rounding errors. We present a
convergence analysis for the generated sequence, establishing bounds on the
convergence rate and attainable accuracy. This theoretical framework covers
quasi-Newton and inexact Newton methods, and is leveraged to propose mixed
precision algorithms. We present a wide set of numerical experiments to
illustrate our theoretical results and the behavior of Newton's method and its
approximate variants in mixed precision floating-point arithmetic.
\end{abstract}


\section{Introduction}\label{sec:intro}

Modern computational science, including fields such as machine learning,
inverse problems, image restoration, and physical simulations, relies on the
efficient solution of complex large-scale optimization problems. These problems
are often ill-conditioned or ill-posed, making them particularly challenging to
solve efficiently. Second-order methods, most notably Newton's method, address
this challenge by exploiting curvature information through the Hessian matrix
of second derivatives. By accounting for the local geometry of the objective
function, they generate search directions that are both appropriately scaled
and well oriented. As a result, second-order methods are often the method of
choice in these settings, offering robust performance together with quadratic
convergence in a neighborhood of the
optima~\cite{nocedalWright,boyd2004convex}.

Despite these
advantages, the adoption of Newton's method is hindered by its high
computational cost per iteration, which is due to the need of
forming the Hessian and solving the resulting (potentially large) linear
system needed to compute the search direction. To mitigate these costs,
various approximations of Newton's method have been developed, including
quasi-Newton methods that replace the Hessian by some
approximations~\cite{martinez2000practicalquasinewton}, and inexact Newton
methods that solve the linear systems
approximately~\cite{dembo1982inexactnewton}.
Nevertheless, even approximate Newton methods remain quite computationally
intensive and so, in this work, we are interested in another promising
and complementary direction to reduce the computational burden of second-order
methods: the use of low precision floating-point arithmetic.
The rise of specialized hardware accelerators, such as 
NVIDIA’s tensor cores~\cite{cuda}, has driven the successful use of
low precision arithmetic to reduce computational and memory costs.  However,
the naive application of low precision can lead to a significant loss of accuracy,
and, in the case of iterative methods, to a slower (or lack of) convergence.  
In order to minimize computational cost while preserving convergence and solution accuracy,
the use of \emph{mixed precision algorithms} has spread in different fields. Notably, many such algorithms have been 
successfully developed in numerical linear algebra; see~\cite{higham2022Survey} for a survey. 
In contrast, the use of mixed precision arithmetic for general nonlinear optimization algorithms
remains largely unexplored and lacks a unifying theoretical framework.

In this work, we address this theoretical gap by proposing a rigorous error
analysis for Newton's method for optimization, accounting for various
sources of inexactness.  Our central contribution is to perform a convergence study of
Newton's method under the assumption that the three main steps of the algorithm
(gradient computation, Hessian-related steps, and iterate update) are computed
inexactly, with possibly different levels of precision.  This study allows for
assessing the impact of the perturbed operations on the final solution accuracy
and on the convergence rate. In turn, this allows for deriving precise
guidelines for assigning a different precision to the various operations.

Importantly,  our error analysis is general enough to encompass a wide range of
approximations, including not only floating-point arithmetic, but also inexact
Newton and quasi-Newton methods (for the latter, we specifically focus on the
Gauss--Newton method). This allows us not only to recover known convergence
results for such variants of Newton's method when all the operations are
performed in exact arithmetic, but also to incorporate rounding errors in such variants. 
This leads us to analyze and propose mixed precision implementations of inexact and quasi-Newton methods, 
whose convergence and accuracy is covered by our theoretical results. Our analysis highlights
the interplay between approximation and rounding errors, and shows how to balance both sources of error.

We illustrate our theoretical findings through extensive numerical experiments,
demonstrating the soundness and generality of our error analysis in
practice and providing useful insights into the behavior of Newton's method and its variants
in mixed precision arithmetic.

\subsection{Related work}

Mixed precision optimization (meaning optimization methods in which different
computations can have different accuracy, or different quantities can be subject to errors of different magnitudes) has been widely studied in the literature. 
  The predominant
approach consists in analyzing the convergence of optimization algorithms in
the presence of inexact function or derivative evaluations \cite{cartis2022evaluation}. 

Many works consider frameworks in which the accuracy of the estimates increases
with time, imposing a decreasing absolute or relative error on the function and
gradient approximations to ensure the convergence properties of the methods
\cite{bellavia2018levenberg,vicente,scheinberg}. Most often these methods are
studied in the finite sum context, suited to machine learning applications,
where the approximations are built by subsampling techniques
\cite{mahoney_sampling,bellavia2023gn,bellavia2020subsampled,roosta2019sub}.
Another line of work inherits techniques from derivative-free optimization and
adopts fully probabilistic frameworks based on the fully linear assumption on
the models employed \cite{vicente,scheinberg,bergou2022stochastic}.  Most of
the works on inexact optimization consider globally convergent methods and
propose a worst-case analysis
\cite{yao2021inexactnewtoncg,bellavia2021adaptive,yao2021inexact}, which counts
the number of iterations necessary to drive the norm of the gradient below a
given threshold; this leads to bounds that are usually pessimistic and rarely
observed in practice. 

We adopt a different perspective here, coming from numerical and error
analysis.  We consider Newton's method without a globalization strategy and we
focus on a deterministic local convergence analysis, to target sharper and more
informative bounds on the convergence of the iterates.  From this perspective
many works have considered root-finding Newton's method, see for
example~\cite{dennis_schnabel,lancaster1996newtonanalysis,wozniakowski1977stabilitynonlineareq,
ypma1983roundingerrorsnewton,ypma1984localinexactnewton,tisseur2001newton,kelley2022newton,
kelley2023threeprec}.  These tighter analyses allow in particular for assessing
much more precisely the effect of rounding errors on Newton's methods.
Not many works in optimization are concerned with this aspect; we mention
~\cite{monnet2023multiprecision,vasin2026lowerupperboundsconvergence}, which
focus on first-order methods,
and~\cite{gratton2019notesolvingnonlinearoptimization}, which, although
studying a second-order optimizer, does not address Hessian approximations. 
    
Our work is closest to Tisseur's~\cite{tisseur2001newton}, which
is focused on root-finding Newton's method. Our
analysis is an adaptation and extension of that in \cite{tisseur2001newton} to the
distinct setting of unconstrained optimization. While the two methods are
intimately related, the analysis for
nonlinear systems does not directly apply to optimization problems, and we believe
that the latter  deserves a proper dedicated analysis.
In particular, specializing the analysis to optimization
problems allows for taking into account sources of errors that
are specific to this context, such as gradient or Hessian approximations. 
Moreover, while adapting Tisseur's analysis to our context, we have also
made some small changes and improvements to make the bounds more readable and slightly sharper.

\subsection{Organization of the paper}

The article is organized as follows. 
\zcref[S]{sec:app_n_conv} presents the error analysis of mixed precision Newton's method in a general framework.
\zcref[S]{sec:n_approx} discusses the application of our framework to floating-point arithmetic, and inexact and Gauss--Newton's methods.
\zcref[S]{sec:numerical_experiments} presents numerical experiments with all these variants.
\zcref[S]{sec:conclusions} provides concluding remarks.

\subsection{Notations}\label{par:notation}

All computed quantities are denoted by a hat. We denote by $\norm{\cdot}$ any
vector norm and the corresponding operator norm, unless otherwise specified, and by $\kappa(A) = \norm{A}\norm{A^{-1}}$ the condition number of a matrix
$A$.   Finally, given a point $x\in \bbR^n$ and a radius $\rho >0$, we
denote by $B_{\rho}(x) = \{y \in \bbR^n : \norm{y - x} < \rho\}$ the open
ball of radius $\rho$ centered in $x$.

\section{Mixed precision Newton's method for optimization}\label{sec:app_n_conv}

Given a twice continuously differentiable function $f: \bbR^n \to \bbR$, a general minimization problem can be stated as follows:
\begin{equation} \label{eq:min_problem}
    \text{minimize } f(x) \text{ w.r.t. } x \in \bbR^n.
\end{equation}

We denote by $g:\mathbb{R}^n\rightarrow \mathbb{R}^n$ the gradient of $f$ and
by $H:\mathbb{R}^n\rightarrow \mathbb{R}^{n\times n}$ its Hessian.  Newton's
method for optimization is an iterative algorithm that generates a sequence of
approximations $\{x_{i}\}$ to a minimizer $x^*$ of $f$ using the following
update rule, assuming that $H(x_i)$ is positive definite:

\begin{equation} \label{eq:newton_update}
    \begin{aligned}
        \text{solve } H(x_i)d_i = - g(x_i), \\
        x_{i+1} = x_i + d_i.
    \end{aligned}
\end{equation}

The behavior of Newton's method in exact arithmetic is well
understood~\cite[Thm.~3.5]{nocedalWright}. If initialized close enough to a
solution, Newton's method generates a well-defined sequence that converges to
that solution at a quadratic convergence rate.  In this section, our goal is to
study the behavior of the method when the operations in~\zcref{eq:newton_update} are subject to errors. These errors may simply be
rounding errors due to finite precision or approximation errors that arise when
inexact or quasi-Newton variants are used. Moreover, we allow different operations to be
affected by errors of different size, and so we refer to this method
as mixed precision Newton's method. 

Specifically, we consider the following error model, 
where computed quantities affected by errors 
are marked by a hat.

\begin{errormodel}\label{def:err_model_newton}
At each iteration $i$, mixed precision Newton's step satisfies
\begin{subequations}
\begin{align}
    &\dhat_i := - \big(H(\xh_i) + \errh\big)^{-1}\big(g(\xh_i) + \errg\big), \label{eq:d}\\
    &\xh_{i+1} =  \xh_i + \dhat_i + \errsum, \label{eq:single_eq_newton}\\
    &\norm{\errg}\le \eg, \label{E^g}\\
    &\norm{ \errh} \leq \eh \norm{ H(\xh_i)}, \label{E^H}\\
    &\norm{ \errsum } \leq \ei \big(\norm{ \xh_i} + \norm{ \dhat_i}\big). \label{E^+}
\end{align}
\end{subequations}
\end{errormodel}

This model depends on three error terms $\errg$, $\errh$, and $\errsum$, which are each bounded as follows.
\begin{itemize}
\item
We bound the norm of the gradient error $\norm{\errg}$ 
with an absolute error of size $\eg$, which is a general term
that accounts for inexactness in the gradient evaluation.
\item
We bound the norm of the Hessian error $\norm{\errh}$
with a relative normwise error of size $\eh$. This accounts
for both the error incurred in forming the Hessian and the backward error for
solving the associated linear system.
\item
We bound the norm of the update error $\norm{\errsum}$
with a relative normwise error of size $\ei$. This accounts for
the error in updating the iterate and rounding it to the working precision.

\end{itemize}
By using the subscript $i$, we account for the nonconstant behavior of
rounding errors, and for potential adaptive precision strategies that vary the precisions
across iterations---even though exploring this setting is outside the scope of this article, it remains
a field of interest for future work. 

We will assume Lipschitz continuity on the Hessian, as defined in~\zcref{def:lip}, which is quite standard to ensure quadratic
convergence~\cite{nocedalWright}.

\begin{definition}\label{def:lip}
    A function $\varphi$ is said to be Lipschitz continuous in an open set $\Omega \subseteq \bbR^n$ if there exists a constant $L \geq 0$, called the Lipschitz constant, such that
    \begin{equation}
        \norm{\varphi(v)-\varphi(w)} \leq L \norm{v-w}, \quad \forall v, w \in \Omega.
    \end{equation}
\end{definition}

Our main result is stated in the following theorem. 

\begin{theorem}\label{th:rel_err}
    Let $f:\mathbb{R}^n\rightarrow\mathbb{R}$ be twice continuously
    differentiable. Let $x^* \in \bbR^n$ be a minimizer for $f$, and assume
    that $H(x^*)$ is nonsingular.
    Let $H$ be Lipschitz continuous with Lipschitz constant $L_H$, as defined in~\zcref{def:lip}, in an open
    neighborhood $\Omega$ of $x^*$. 
    Let $\{\xh_{i}\}$ be the sequence generated by mixed precision Newton's method under \zcref{def:err_model_newton}. If, at some iteration $i$,
    \begin{equation} \label{def:nu_i_thm}
        \nu_i := \eh \kappa(H(\xh_i)) < 1, 
    \end{equation}
    then 
    \begin{equation} \label{thes2.1}
    \norm{ \xh_{i+1} - x^*} \leq \alpha_i\norm{\xh_i - x^*}^2 + \beta_i\norm{\xh_i - x^*} + \limacc,
    \end{equation}
    where 
    \begin{align}
        \alpha_i &:= \frac{1+\ei}{2(1-\nu_i)}L_H\norm{ \Hi^{-1} }, \label{def:alpha_i_thm} \\
        \beta_i &:= \frac{(\eh + \ei)}{1-\nu_i}\kappa(\Hi) + \ei, \label{def:beta_i_thm} \\
        \limacc &:= \frac{1+\ei}{1-\nu_i}\eg \norm{ \Hi^{-1}} + \ei \norm{ x^*}, \label{def:limacc_thm}
    \end{align}
    Moreover, if \zcref{def:nu_i_thm} holds for all $i$ and 
    \begin{equation} \label{def:theta_i_thm}
        \theta_i := \alpha_i\norm{\xh_i - x^*} + \beta_i < \theta_{\max}
    \end{equation}
    for  $\theta_{\max} \in [0, 1)$, then there exists $\rho > 0$ such that if $\xh_0 \in B_{\rho}(x^*)$, the sequence $\{\xh_{i}\}$ is well defined and satisfies \zcref{thes2.1} for all $i$ until 
    \begin{equation}
        \norm{\xh_i - x^*} < \frac{\limacc}{1-\theta_{\max}}.
    \end{equation}
\end{theorem}

\begin{proof}
See~\zcref[]{sec:proof_rel_err}. 
\end{proof}

\zcref{th:rel_err} shows that the error decreases until the first
iteration for which it becomes smaller than $\gamma_i/(1-\theta_{\max})$. Assuming $\theta_{\max}$ is sufficiently less than $1$, which is satisfied when $L_H\norm{ \Hi^{-1} }\norm{\xh_i - x^*}$ and $\nu_i$ are safely below one, 
this means that the quality of the possible approximation of $x^*$ is mainly determined by
$\limacc$,
which we call the \textit{limiting accuracy}.
To first order,
\begin{align} 
    \theta_i &\approx (\eh + \ei)\kappa(\Hi) + \frac{1}{2}L_H\norm{ \Hi^{-1} } \norm{ \xh_i - x^*}  + \ei, \label{eq:theta_approx} \\
    \limacc &\approx \eg \norm{ \Hi^{-1}} + \ei \norm{ x^*}. \label{eq:limacc_approx}
\end{align}
Hence, the limiting accuracy depends, to first order, on the gradient error $\eg$ and on the working precision $\ei$, but not
on the Hessian error $\eh$. This shows that we may tolerate
some errors in forming the Hessian and solving the associated linear system without
impacting the final solution accuracy. Moreover, since $\eg$ 
depends on the error in evaluating the gradient, which may be large, 
and since $\eg$ is multiplied by $\norm{\Hi^{-1}}$, \zcref{eq:limacc_approx} suggests that the error incurred in the gradient evaluation
is the one that impacts the limiting accuracy the most. Therefore,
if we wish to obtain the highest possible solution quality, the gradient
should be evaluated as accurately as possible.
Finally, note that the behavior of $\limacc$ over the iterations is not necessarily monotone, especially due to
$\eg$, which can change over time depending on the conditioning of the gradient at the current iterate $\xh_i$.

Let us now turn to the convergence rate of the method. Inequality
\zcref{thes2.1} shows that it contains both a quadratic term $\alpha_i$ 
(see \zcref{def:alpha_i_thm})
and a linear term $\beta_i$ (see \zcref{def:beta_i_thm}).
To first order,  these terms behave as
\begin{equation} \label{eq:alpha_beta_approx}
    \alpha_i \approx \frac{1}{2}L_H\norm{ \Hi^{-1}}, \quad \beta_i \approx \eh \kappa(\Hi).
\end{equation}
First, in exact arithmetic, and in absence of approximations on the gradient and the Hessian, we have $\ei = \eg = \eh = \nu_i = 0$ for all $i$,
and so $\beta_i = \gamma_i=0$. Using the same argument as in~\cite[Thm.~3.5]{nocedalWright}, we have that, if $\xh_i$ is close enough to the solution, $\norm{\Hi^{-1}} \leq 2\norm{\Hstar^{-1}}$, recovering the standard Newton's quadratic convergence, under the same assumptions.
Moreover, even in presence of errors, Newton's method may still converge quadratically
if the linear term $\beta_i$ is small compared with the quadratic one $\alpha_i$; 
conversely, if $\beta_i$ becomes dominant, then this will
deteriorate the convergence to a linear rate, or may even prevent convergence if $\beta_i \ge 1$.
Importantly, $\beta_i$ depends on the Hessian error, which shows
that forming and solving the Hessian system approximately will affect the convergence rate of the method. Finally,
in order to preserve convergence, the level of error $\eh$ introduced in the Hessian should be chosen to be inversely proportional
to the condition number of the Hessian $\kappa(\Hi)$: the more ill-conditioned the Hessian, the smaller the tolerated error.

\subsection{Analogous result for the gradient norm convergence}

We conclude this section by proving an analogous result for the gradient norm.

\begin{theorem} \label{th:grad_norm}
    Let $f:\mathbb{R}^n\rightarrow\mathbb{R}$ be twice continuously
    differentiable. Let $x^* \in \bbR^n$ be a minimizer for $f$, and assume
    that $H(x^*)$ is nonsingular. Let $H$ be Lipschitz continuous with Lipschitz
    constant $L_H$, as defined in~\zcref{def:lip}, in an open neighborhood
    $\Omega$ of $x^*$. Let $\{\xh_{i}\}$ be the sequence generated by mixed
    precision Newton's method under \zcref{def:err_model_newton}. If \zcref{def:nu_i_thm} holds for some iteration $i$, 
    then
    \begin{equation} \label{thes2.2}
        \norm{\ginext} \leq \phi_i\norm{\gi} + \limg,
    \end{equation}
    where
    \begin{align}
        \phi_i &:= \frac{1}{1-\nu_i} \Bigl((\eh + \ei) \kappa(\Hi) + \frac{1+\ei}{2}\big((1 + \theta_i)\mu_i + \tau_i\big)\Bigr) \label{def:phi_i_thm} \\
        \limg &:= \eg \Bigl( 1 + \tau_i + \frac{1}{1- \nu_i}\bigl( (\eh + \ei) \kappa(\Hi) + (1+\ei)(1+\theta_i)\mu_i/2 \bigr)\Bigr)  \nonumber \\
        &\qquad + \ei \norm{\Hi}\norm{\xh_i}\bigl(1 + \tau_i/2 + (1 + \theta_i)\mu_i/2\bigr), \label{def:limg_thm}
    \end{align}
    where $\mu_i := L_H\norm{\Hi^{-1}}\norm{\xh_i - x^*}$, $\tau_i := L_H
    \norm{\Hi^{-1}} \limacc$, and $\limacc$ is defined
    in~\zcref{def:limacc_thm}. 

    Moreover, if \zcref{def:nu_i_thm} holds for all $i$ and $\phi_i <
   \phi_{\max} \in [0, 1)$, then there exists $\rho > 0$ such that if $\xh_0 \in
   B_{\rho}(x^*)$, the sequence $\{g(\xh_{i})\}$ is well defined and satisfies
   \zcref{thes2.2} for all $i$ until
    \begin{equation}
        \norm{\gi} < \frac{\limg}{1-\phi_{\max}}.
    \end{equation}
\end{theorem}

\begin{proof}
See~\zcref{sec:proof_grad_norm}. 
\end{proof}

As for~\zcref{th:rel_err}, the gradient norm decreases until the first iteration for which it becomes smaller than $\limg/(1-\phi_{\max})$. To first order, $\phi_i$ and $\limg$ behave as
\begin{align}
    \phi_i &\approx \frac{1}{2}L_H\norm{ \Hi^{-1}}\big(\norm{ \xh_i -x^*}+\limacc\big) + (\ei + \eh)\kappa(\Hstar), \label{eq:phi_i_approx}\\
    \limg &\approx \eg + \ei\norm{ H(\xh_i)} \norm{ \xh_i}.\label{eq:limg_approx}
\end{align}

The assumptions of this second theorem are stronger than the ones
of~\zcref{th:rel_err}, since they also require a bound on the term $\limacc$.
This assumption is necessary to ensure that the absolute error on the
solution decreases enough to show a decrease in the gradient norm too. Moreover,  in
classical Newton's method, the gradient norm decreases quadratically. In
this theorem the term responsible for the quadratic convergence is hidden in $\phi_i$,
which contains a factor $\norm{ \xh_i - x^*}$, which multiplies $\|\gi\|$. 
    
Additionally, we see that the error $\eg$ on the gradient evaluation impacts not only $\limg$, but also the convergence rate $\phi_i$ of the gradient norm, through the term $\tau_i$.
This contrasts with \zcref{th:rel_err}, in which $\eg$ does not affect the convergence rate of the error on the solution $\norm{\xh_i - x^*}$, but only its limiting accuracy.

\section{Newton's approximations}\label{sec:n_approx}

\zcref{def:err_model_newton} is quite general and encompasses
different sources of errors. First of all, it covers the rounding errors
arising from the use of finite precision floating-point arithmetic;
we consider these errors in~\zcref{sec:mp_n}. However, it can account for
more general sources of errors, typically arising from
approximations introduced to make classical Newton's method more suitable for large-scale problems.
Specifically, we consider inexact Newton (\zcref{sec:inexact_newton})
and Gauss--Newton (\zcref{sec:gauss_newton}) methods; for both of these variants,
we not only discuss the errors introduced by their approximations,
but also their interplay with rounding errors, that is, we consider mixed precision
inexact Newton and Gauss--Newton in floating-point arithmetic.

\subsection{Floating-point Newton} \label{sec:mp_n}

In any floating-point arithmetic compliant with the IEEE 754 standard~\cite{IEEE754-2019}, 
the elementary operations satisfy the following model~\cite[sect.~2.2]{highamAccuracy}: 
\begin{equation}\label{FPAmodel}
  \fl(a\op b) = (a\op b)(1+\delta),\quad \abs{\delta}\le u,\quad \op\in\{+, -, \times, /\},
\end{equation}
where $u$ is the unit roundoff of the precision used and $\fl(\cdot)$ represents the results computed 
in floating-point arithmetic.
Hence, floating-point arithmetic introduces relative errors proportional
to the unit roundoff of the arithmetic.

In our context, we consider a mixed precision approach for Newton's method,
outlined in~\zcref{alg:base_mp_newton}, 
which uses three floating-point arithmetics with
different unit roundoffs:
\begin{itemize}
\item $u$ is the unit roundoff of the working precision, used for storing and updating the iterates;
\item $u_g$ is the unit roundoff of the arithmetic used for evaluating the gradient;
\item $u_H$ is the unit roundoff of the arithmetic used for forming and solving the Hessian system.
\end{itemize}

\begin{algorithm}[ht!]
\caption{Mixed precision Newton}
\label{alg:base_mp_newton}
\hspace*{\algorithmicindent} \textbf{Input}: initial guess $x_0$, Hessian $H$, gradient $g$ \\ 
\hspace*{\algorithmicindent} \textbf{Output}: an approximation $x_{i+1}$ to the minimizer $x^*$
\begin{algorithmic}[1]
\For{$i = 0, 1, \ldots$ until convergence}
    \State Compute $g_i = g(x_i)$      \tabto{4cm} in precision with unit roundoff $u_g$
    \State Solve $H(x_i) d_i = -g_i$   \tabto{4cm} in precision with unit roundoff $u_H$ \label{line.Hessian}
    \State Update $x_{i+1} =x_i + d_i$ \tabto{4cm} in precision with unit roundoff $u$
\EndFor
\State \textbf{return} $x_{i+1}$
\end{algorithmic}
\end{algorithm}

Let us now discuss how the unit roundoffs $u$, $u_H$, $u_g$ relate to the corresponding
error terms $\ei$, $\eh$, $\eg$ in~\zcref{def:err_model_newton}.
By~\zcref{FPAmodel}, we readily have $\ei = u$.
For the Hessian system, a backward stable solver 
will deliver an error $\eh$ of order $u_H$; for example, for a direct solver based on
Cholesky factorization, $\eh = O(n^2)u_H$~\cite[eq.~(10.7)]{highamAccuracy},
where the dimensional constant in $O(n^2)$ is known to be pessimistic~\cite{hima19a}.
Finally, $\eg$ will be a (potentially large) multiple of $u_g$, but its precise 
value is very much dependent on the expression of the gradient,
the point at which it is evaluated, and the method of evaluation.
We will discuss in~\zcref{subsec:experimental_setup} how to measure
these errors in practice.

Given the discussion in the previous section,
the setting of interest
is $u_g \le u \le u_H$: we consider the use of a potentially
higher precision to evaluate the gradient (to improve the limiting accuracy)
and of a potentially lower precision to form and solve the Hessian system (to
reduce the computational cost, while preserving high limiting accuracy, at the price of
potentially deteriorating the convergence rate).

\subsection{Inexact Newton} \label{sec:inexact_newton}

Newton's method requires solving a linear system of the form $\Hi\dhat_i = -\gi$ at each iteration.
For large-scale problems, solving this system exactly with a 
direct method can be quite expensive. Instead, inexact Newton's methods~\cite[chap.~7.1]{nocedalWright} solve this system 
approximately by an iterative solver
 such as the conjugate gradient (CG) method~\cite{hestenes1952cgls}.
The standard criterion to stop such an iterative method is to stop whenever the computed $\dhat_i$
satisfies, for given tolerances $0\leq \eta_i<1$,
\begin{equation} \label{stopping}
\norm{\Hi\dhat_i+\gi}\leq \eta_i \norm{\gi}.
\end{equation}

Our framework can be applied to inexact Newton's method using
the Rigal--Gaches theorem~\cite[Thm.~7.1]{highamAccuracy}, which shows that the following two statements are equivalent~\cite[eq.~(17.33b)]{highamAccuracy}:
\begin{enumerate}
\item $\exists \errh : (\Hi+\errh)\dhat_i = -\gi, \quad \norm{\errh} \le \eh\|\Hi\|$;
\item $\norm{\Hi\dhat_i + \gi} \le \eh\norm{\Hi}\norm{\dhat_i}$.
\end{enumerate}
Hence inexact Newton's method 
in exact arithmetic and with no approximations on the gradient  
satisfies \zcref{def:err_model_newton}
with 
\begin{equation} \label{def:ki}
\eh = \eta_i \frac{\norm{\gi}}{\norm{\Hi}\norm{\dhat_i}} =: \frac{\eta_i}{\ki},
\end{equation}
where $\ki = \norm{\Hi}\norm{\dhat_i}/\norm{\gi}$ satisfies
\begin{equation}
\frac{1}{1+\eh} \le \ki \le \frac{\kappa(\Hi)}{1-\eh \kappa(\Hi)}.
\end{equation}

In the regime where, due to the inexactness in the linear system solution, the term $\beta_i$ dominates over the term $\alpha_i$, 
\zcref{th:rel_err} proves a convergence rate of the form
\begin{equation}\label{eq:rateINours}
    \begin{aligned}
        &\norm{ \xh_{i+1} - x^*} \le \beta_i \norm{ \xh_i -x^* }, \\
        &\beta_i \approx \eh\kappa(\Hi) = \eta_i \frac{\kappa(\Hi)}{\ki}.
    \end{aligned}
\end{equation}
This can be compared with the standard convergence theory of inexact Newton's
method, as in~\cite{dembo1982inexactnewton}, \cite[chap.~7.1]{nocedalWright}, where a convergence rate is proved in the energy norm induced by $H(x^*)^2$ with $\beta_i = \eta_\mathrm{max}$
for any $\eta_\mathrm{max}$ such that $\forall i, \eta_i < \eta_\mathrm{max}$,
which implies $\norm{ \xh_{i+1} - x^*} \le \eta_{\max} \kappa(\Hi)\norm{ \xh_i -x^* }$.

Exploiting the generality of our framework,  we can combine inexact Newton's method
with mixed precision floating-point arithmetic. This amounts to modifying
\zcref{alg:base_mp_newton} so that the Hessian system on line~\ref{line.Hessian}
is solved by an iterative solver with stopping tolerance $\eta_i$. 
Then \zcref{def:err_model_newton} is satisfied with $\eh \leq c_\mathrm{solver} u_H + \eta_i/\ki$,
where the constant $c_\mathrm{solver}$ depends on the specifics of the iterative
solver used; for CG, $c_\mathrm{solver}=O(nk^2)$, where $k$ is
the number of iterations~\cite{bake2025forwardbackwarderrorbounds}.
This gives some indication on how to choose $\eta_i$ based on the unit roundoff $u_H$
of the arithmetic used to solve the Hessian system (or vice versa): in order to equilibrate both
sources of inexactness, we should set $\eta_i \approx \ki c_\mathrm{solver} u_H$.
We will illustrate this rule of thumb experimentally in \zcref{subsec:in_exp}.

\subsection{Gauss--Newton} \label{sec:gauss_newton}

Our error model~\zcref{def:err_model_newton} can potentially encompass
quasi-Newton methods. Indeed, the term $\errh$ can be used to represent the errors
arising from the approximation of the Hessian matrix.  In this section, we
present as an example the Gauss--Newton method, a quasi-Newton method
specifically designed for nonlinear least-squares problems, that is, problems of
the form
\[
\text{minimize } f(x) \text{ w.r.t. } x \in \bbR^n\text{, with } f(x)=\frac{1}{2}\norm{r(x)}^2,
\]
where $r: \bbR^n \to \bbR^m$ is referred to as the \textit{residual}.
Exploiting the problem's structure, the method builds a Hessian approximation
only using first-order information of $r$ \cite{dennis_schnabel}. In fact, the gradient of 
$f$ can be expressed as $g(x) = J(x)^Tr(x)$, with  $J: \bbR^n \to \bbR^{m
\times n}$ the Jacobian matrix of $r$, and  the second-order derivatives as
\begin{equation}
    H(x) = J(x)^TJ(x) + S(x), \quad 
    S(x) = \sum_{i=0}^{m-1} r_i(x) \nabla^2 r_i(x),
\end{equation}
where $r_i(x)$ 
is the $i$th component of the residual $r(x)$. The Gauss--Newton method
approximates the Hessian $H(x)$ with $J(x)^TJ(x)$, thus neglecting the term
$S(x)$, which contains the second-order derivatives of the residual. Its iterations thus read:
\begin{equation} \label{GN_step}
    \begin{aligned}
\text{solve } \;\Ji^T \Ji \dhat_i &= -\Ji^T r(\xh_i), \\
\widehat{x}_{i+1} &=\xh_i+\dhat_i.
    \end{aligned}
\end{equation}

The convergence of the Gauss--Newton method depends on the relative importance
of the discarded term $\Si$ with respect to $\Ji^T\Ji$. If this term is negligible,
we can recover the fast quadratic convergence of Newton's method, but, if it
is not, the convergence can degrade to a linear one or the method may not converge at all.
This result, for exact arithmetic, can for instance be found 
in~\cite[Thm.~10.2.1]{dennis_schnabel}. 

It is possible to apply \zcref{def:err_model_newton} to the
Gauss--Newton method by interpreting the discarded term $\Si$ as a perturbation
matrix $\errh$.
Defining $\errh = -\Si$, we have $\eh = \norm{\Si}/\norm{\Hi}$
and \zcref{th:rel_err} applies with $\alpha_i\approx L_H \|\Hi^{-1}\|$ and $\beta_i \approx  \eh\kappa(\Hi) = \norm{\Si}\norm{\Hi^{-1}}$. 
We can compare this convergence result to the one in \cite[Thm.~10.2.1]{dennis_schnabel}. The latter is based on the key assumption that there exists a $\sigma$ such that 
\begin{equation} \label{eq:gn_standard_assumption}
    \norm{\Sstar (\xh_i - x^*)} \leq \sigma \norm{\xh_i - x^*} < \lambda_\mathrm{min}(\Jstar^T\Jstar)\norm{\xh_i - x^*},
\end{equation}
where $\lambda_\mathrm{min}(\cdot)$ denotes the smallest eigenvalue of a matrix.
Since the first inequality is certainly satisfied with $\sigma=\norm{\Sstar}$, a sufficient condition for this assumption to hold
is $\norm{\Sstar} < \lambda_\mathrm{min}(\Jstar^T\Jstar)$.

We can relate this assumption to ours as follows, using the 2-norm, denoted as $\|\cdot\|_2$. Our assumptions in~\zcref{def:err_model_newton} and \zcref{th:rel_err} require
\[
\epsilon_i^H=\frac{\|\Si\|_2}{\|\Hi\|_2} < \frac{1}{\kappa(\Hi)}=\frac{1}{\|\Hi\|_2\|\Hi^{-1}\|_2},
\]
and thus 
\[
\|\Si\|_2 < \frac{1}{\|\Hi^{-1}\|_2}=\; \lambda_{\min}(\Hi). 
\]
By Weyl’s inequality~\cite[Thm.~4.3.1]{horn1985matrixanalysis}, for all $x$ such that $S(x)$ is positive semidefinite, it holds
\[
\lambda_{\min}(H(x))\geq \lambda_{\min}(J(x)^TJ(x))+ \lambda_{\min}(S(x))\geq \lambda_{\min}(J(x)^TJ(x)).
\]
Thus if $S(x^*)$ is positive semidefinite and 
$\|S(x^*)\|_2<\lambda_{\min}(\Jstar^T\Jstar)$, then also $\|S(x^*)\|_2< \lambda_{\min}(H(x^*))$.

Concerning the  convergence rate,  in \cite[Thm.~10.2.1]{dennis_schnabel} we have  for all $i$
\[
 \alpha_i\approx\frac{\|\Jstar\|_2L_J}{\lambda_{\min}(\Jstar^T\Jstar)}, \qquad
 \beta_i\approx \frac{\sigma}{\lambda_{\min}(\Jstar^T\Jstar)},
\]
 assuming $J(x)$ to be $L_J$ -Lipschitz, as defined in~\zcref{def:lip}.
Assuming again $S(x^*)$ to be positive semidefinite, and that it is $L_S$-Lipschitz, our rates satisfy
\begin{align*}
\alpha_i&\approx \frac{L_H}{2} \|\Hi^{-1}\|_2\leq \frac{L_H}{2\lambda_{\min}(\Ji^T\Ji)}\leq \frac{2\sup_x\|J(x)\|_2L_J+L_S}{2\lambda_{\min}(\Ji^T\Ji)},\\
\beta_i&\approx \|\Si\|_2\|\Hi^{-1}\|_2=\frac{\|\Si\|_2}{\lambda_{\min}(\Hi)}\leq\frac{\|\Si\|_2}{ \lambda_{\min}(\Ji^T\Ji)}.
\end{align*}
We thus obtain a convergence rate bound similar to \cite{dennis_schnabel} when $\widehat{x}_i$ approaches $x^*$.

Once again, note that our framework accounts for both the Gauss--Newton
approximation of the Hessian and any other source of inexactness, in particular
the use of mixed precision floating-point arithmetic. We can obtain such a
mixed precision Gauss--Newton method by modifying \zcref{alg:base_mp_newton} by
replacing the Hessian matrix $\Hi$ on line~\ref{line.Hessian} by $\Ji^T\Ji$.
Then \zcref{def:err_model_newton} is satisfied with $\eh = c_\mathrm{solver}
u_H + \norm{\Si}/\norm{\Hi}$, where $c_\mathrm{solver}$ is a constant depending on the method used for solving the Gauss--Newton linear system.  
This shows that Gauss--Newton can be quite resilient
to the use of low precision for the Hessian, since the $u_H$ term will only impact the convergence
rate if it is dominant compared to $\norm{\Si}$. Conversely, Gauss--Newton may
perform just as well as Newton when using low precision if the rounding errors dominate.
We will illustrate this observation experimentally in \zcref{subsec:in_exp}.

\section{Numerical Experiments}\label{sec:numerical_experiments}

In this section, we present numerical experiments to validate our
bounds and illustrate the conclusions that we can draw from them. 
The code used to perform these experiments is
available online\footnote{\url{https://gitlab.inria.fr/gcarrino/mpnewton}}. 
After describing our experimental setting 
in~\zcref{subsec:experimental_setup}, we
focus first on standard Newton's method in~\zcref{subsec:pure_newton},
and then consider the inexact Newton
and Gauss--Newton variants in \zcref{subsec:in_exp} and \zcref{subsec:gn_exp}, 
respectively.

\subsection{Experimental setting} \label{subsec:experimental_setup}

We now describe the setup used throughout this section. We first outline the implementation details, including the floating-point arithmetics and precision combinations considered, and then introduce the test problems used to validate our theoretical bounds.

\subsubsection{Implementation}

The algorithms have been implemented using python, leveraging the NumPy
library to use different floating-point arithmetics. We consider the standard double (fp64) and single (fp32) precisions, 
as well as bfloat16 (abbreviated bf16 in the charts) arithmetic, simulated via the library ml\_dtypes\footnote{\url{https://pypi.org/project/ml-dtypes/}}.
We also use extended precision in order to compute the reference solution $x^*$; we use np.float128, though
that does not provide quadruple precision as the name suggests, but rather an 80-bit ``long double'' precision. We use the notation $u\equiv$ fp32 to indicate that the precision with unit roundoff $u$ has been set to fp32 (for example).

We will consider different
precision combinations, denoted as tuples $(\cdot, \cdot, \cdot)$, indicating,
respectively, the precisions with unit roundoff $u_g$, $u$, and $u_H$ in \zcref{alg:base_mp_newton}.

For almost all the experiments, we plot the convergence history of both the
relative error $\norm{\xh_i - x^*}/\norm{x^*}$ and the gradient norm
$\norm{\gi}$ (always on the left and right part of the figures, respectively).

We also plot the bound on the relative error and the
gradient norm, as derived in~\zcref{th:rel_err} and~\zcref{th:grad_norm} respectively.
We use slightly more transparent, dashed curves for these bounds and only plot them when the required assumptions
are satisfied.
In order to compute these quantities,
we must compute both $\eg$ and $\eh$. For $\eg$,
we compute $\gi$ in precision
np.float128, and set $\eg$ to the norm of the difference between $\gi$ and the gradient
computed in the chosen precision with unit roundoff $u_g$.
The error on the Hessian $\eh$ is computed using the Rigal--Gaches formula~\cite[Thm.~7.1]{highamAccuracy} for the backward error, evaluated in precision np.float128.

\subsubsection{Test problems} \label{subsec:test_problems}

For most experiments, we will use the following two test problems. Note that,
in this section only, $x_k$ refers to the $k$th component of the vector $x$,
not the $k$th iterate of the algorithm.
\begin{itemize}
\item \texttt{ENGVAL1}\footnote{\url{https://vanderbei.princeton.edu/ampl/nlmodels/cute/engval1.mod}}: a standard
minimization problem from the CUTEst dataset~\cite{gratton2024s2mpjcutestoptimizationproblems};
the function to be minimized is
\begin{equation} \label{eq:engval1}
    f(x) = 3+\sum_{k=0}^{n-2} (x_k^2+x_{k+1}^2)^2 -4x_k,
\end{equation}
where $x \in \bbR^n$ and $n=100$ in our setting.
\item \texttt{SINREG}:
a least-squares regression problem on some syntethic vectors of datapoints $z, y \in \bbR^m$, where the function to be minimized is:
\begin{equation} \label{eq:regression_base}
    f(x) = \frac{1}{2}\norm{F(x, z) - y}^2 = \frac{1}{2}\norm{r(x)}^2,
\end{equation}
and $F : \mathbb{R}^{n} \times \mathbb{R}^m \to \mathbb{R}^m$ is a model parametrized by $x \in \mathbb{R}^n$.
In our experiments, we use $m=50$ and 
$F$ defined as
\begin{equation} \label{eq:d_dim_model}
    F(x, z) = x_0 z + \sum_{k=1}^{\left\lceil\frac{n-1}{2}\right\rceil} x_{2k-1} z^{k+1} + \sum_{k=1}^{\left\lfloor\frac{n-1}{2}\right\rfloor}x_{2k} \sin(x_{2k} z),
\end{equation}
where the power in $z^{k+1}$ is applied componentwise.
This test problem has a diagonal Hessian whose condition number can be easily 
controlled. In each set of experiments, we define a reference solution
$\bar{x}^*$ and $y$ is computed as $F(\bar{x}^*,z) + \xi$, where
$\xi \in \bbR^m$ is noise randomly sampled from a
uniform $[-\delta, \delta]$ distribution. The value of $\delta$ is $10^{-1}$, unless otherwise specified.
\end{itemize}
For both problems, since no solution $x^*$ is available beforehand, $x^*$ is set to the solution found by standard
Newton's method in uniform extended precision (that is, with all operations performed in np.float128)
after at most $500$ iterations. 

We will also perform experiments on a wider range of CUTEst
problems in \zcref{subsec:cute}.

\subsection{Standard Newton in mixed precision}
\label{subsec:pure_newton}

In this section we consider standard Newton's method in mixed precision floating-point arithmetic,
as described in \zcref{alg:base_mp_newton}.
In this first set of experiments, we do not want the Hessian system solution to
be affected by any inexactness other than floating-point errors; thus the
linear systems are solved directly using LU
factorization.
Hence, in this setting we have $\eh \approx u_H$.

\subsubsection{Floating-point errors}

We consider the \texttt{SINREG} problem as defined in~\zcref{eq:regression_base}--\zcref{eq:d_dim_model}, because it allows us to study the impact of the condition number of the Hessian. 
Indeed, 
choosing\footnote{We choose the entries of $\bar{x}^*$ to be non-integer decimal numbers in order to avoid its floating-point representation to be exact.} 
$n=4$ and $\bar{x}^* = (1.0123, 2.01234, 1.01231, 2.01234)$, we have $\kappa(\Hstar) \approx 2 \times 10^1$.

\begin{figure}[ht]
\centering
    \begin{minipage}{.5\textwidth}
        \centering
        \includegraphics[width=.95\linewidth]{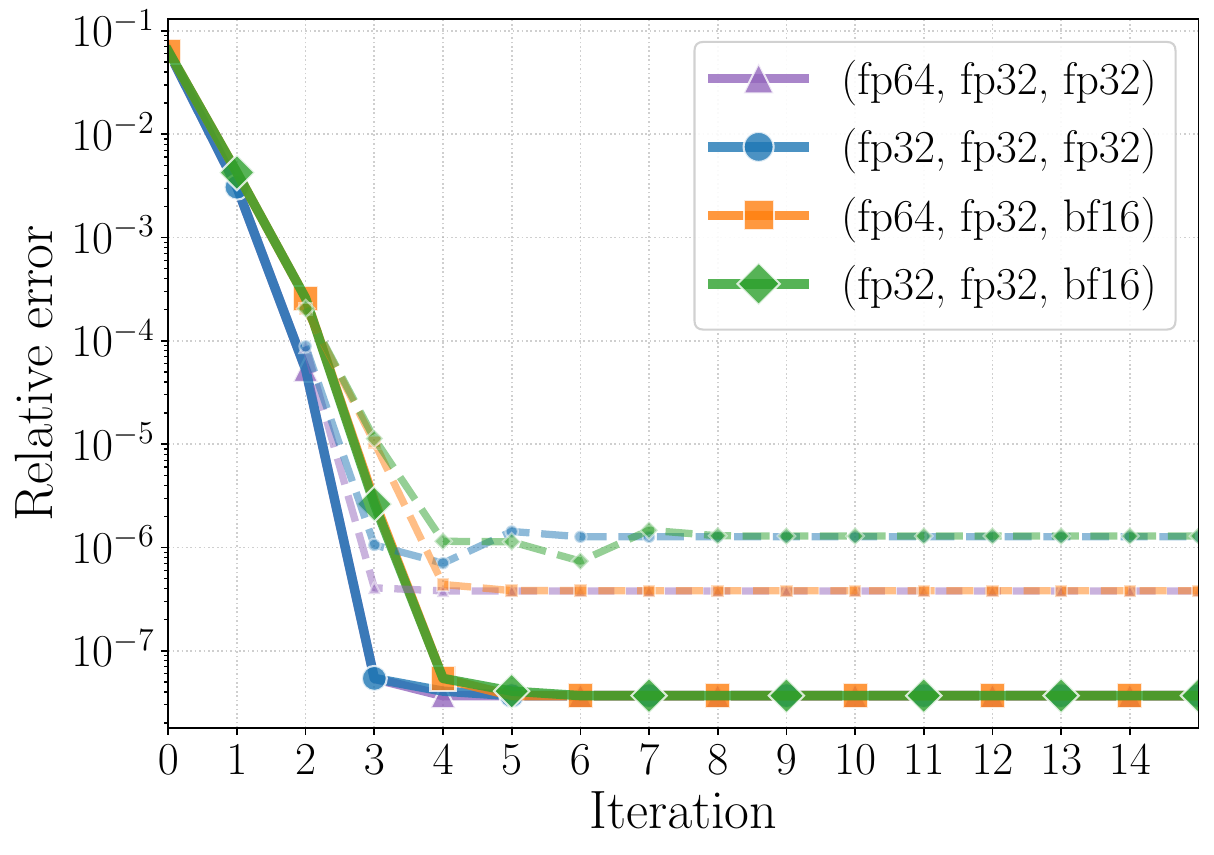}
    \end{minipage}%
    \begin{minipage}{.5\textwidth}
        \centering
        \includegraphics[width=.95\linewidth]{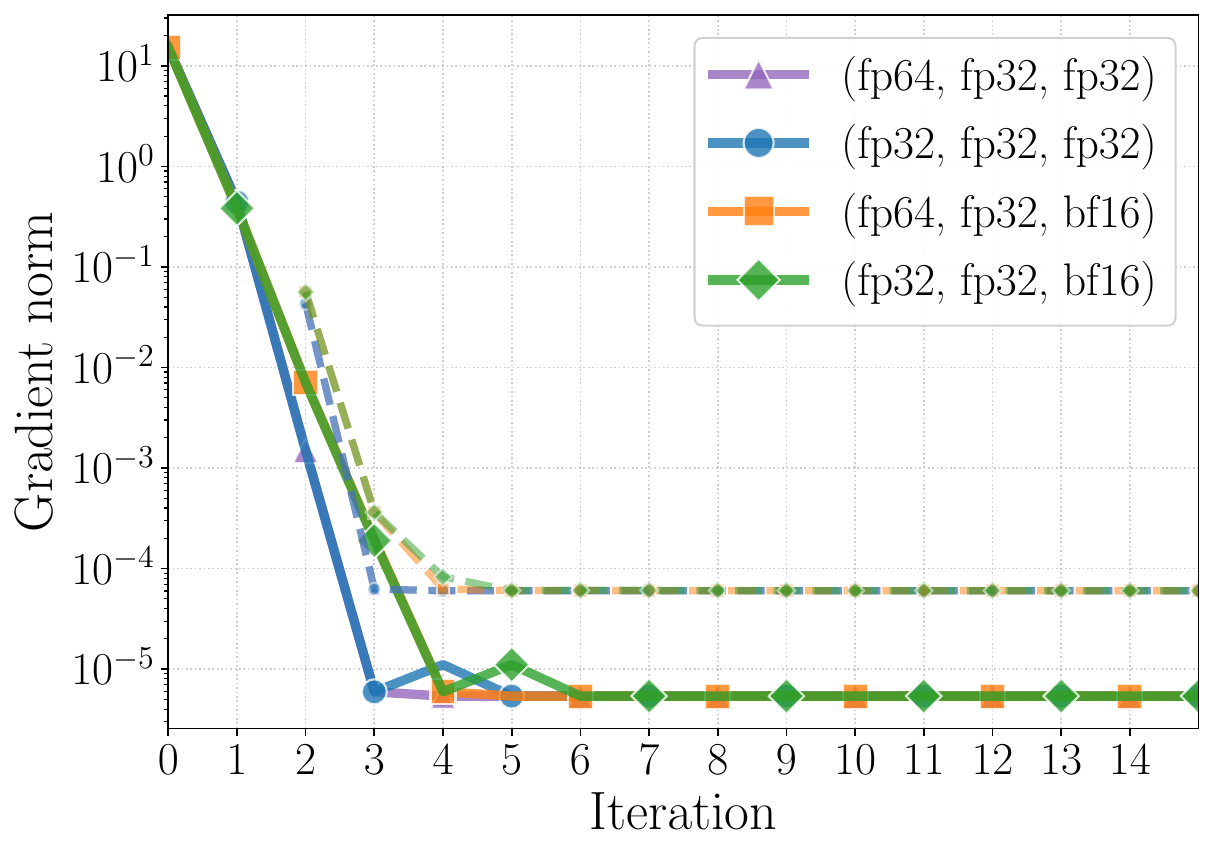}
    \end{minipage}
   \caption{Convergence of mixed precision Newton in relative error (left) and
   gradient norm (right), for different precision sets with unit roundoffs
   $(u_g,u,u_H)$, on the \texttt{SINREG} problem with $n=4$ and $x_0 =
   \bar{x}^* + 10^{-1}$. 
   The purple curves mostly overlap the blue ones, and
   the orange curves mostly overlap the green ones.}
   \label{fig:3}
\end{figure}

\zcref{fig:3} shows the convergence of mixed precision Newton for this problem.
The assumptions of~\zcref{th:rel_err} and~\zcref{th:grad_norm} are easily satisfied in this case, even for low precisions.
Therefore, the method converges for all precision combinations and respects the theoretical bounds.
Moreover, these theoretical bounds (transparent dashed curves) are quite
descriptive as they capture well the actual convergence behavior (solid
curves), for all precision combinations. 

Comparing the purple and blue curves (for which $u_H\equiv$ fp32) with the orange and green
ones (for which $u_H\equiv$ bfloat16), we can see that $u_H$ does not impact the limiting accuracy and
only slightly impacts the convergence rate, since the problem is well conditioned. 

Moreover, the purple and blue curves mostly overlap and, similarly, the orange and green curves also mostly overlap. 
This shows that both the convergence rate and limiting accuracy 
are essentially unchanged whether
we set $u_g\equiv$ fp32 or $u_g\equiv$ fp64. In particular, the attainable
relative error is mainly determined by the working precision $u$ in these charts. 
Further experiments, that we omit for brevity, however confirm that it is actually
regulated by $\max\left(\eg \norm{ H(x^*)^{-1}}, u \norm{ x^*}\right)$, 
as predicted by the theory through $\limacc$ in~\zcref{eq:limacc_approx}. 
In this case, the gradient is computed analytically, its evaluation is not
significantly affected by propagation of rounding errors, and the Hessian
matrix is well conditioned, and so the contribution of $u$ dominates. In this
context, then, we can just use two different precisions,  
computing the gradient in the same precision as the target working precision ($u_g=u$)
and using a lower precision for the Hessian ($u_H \gg u$).

\begin{figure}[ht]
\centering
    \begin{minipage}{.5\textwidth}
        \centering
        \includegraphics[width=.95\linewidth]{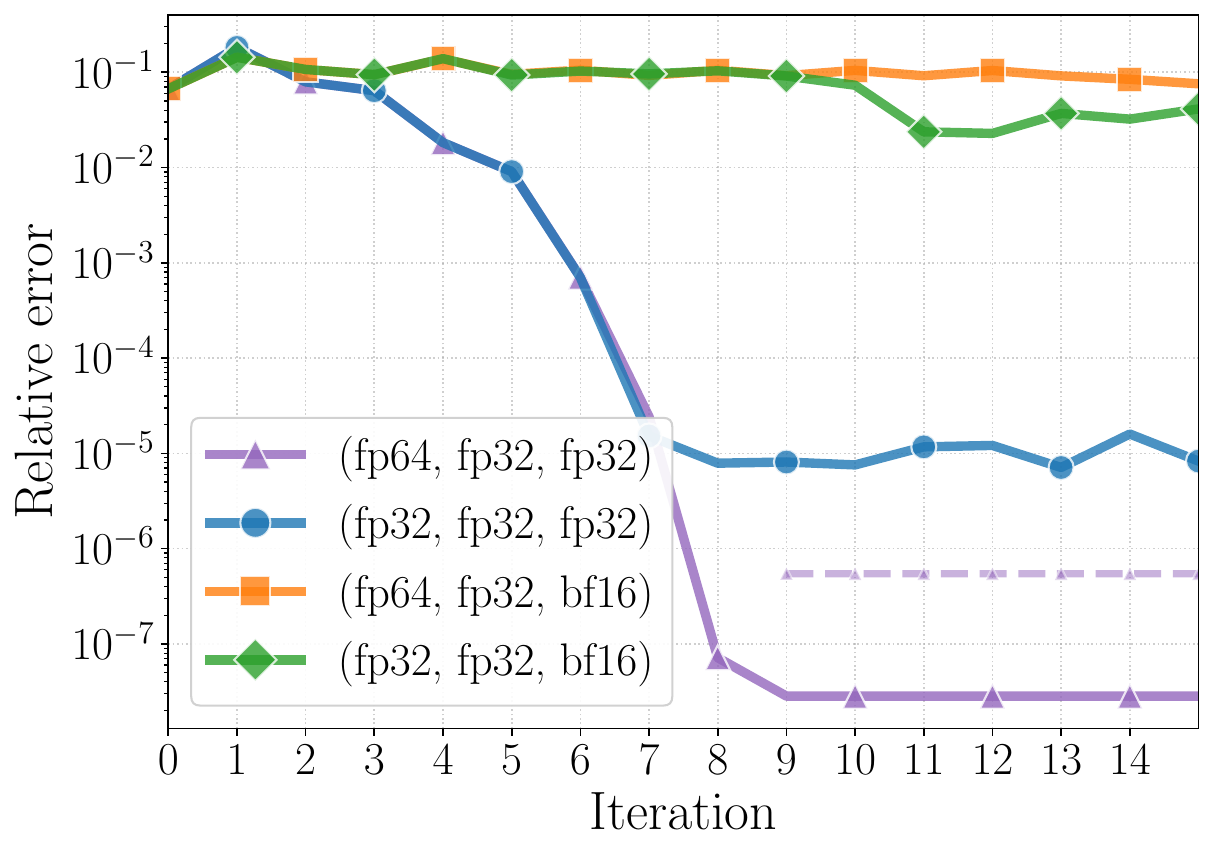}
    \end{minipage}%
    \begin{minipage}{.5\textwidth}
        \centering
        \includegraphics[width=.95\linewidth]{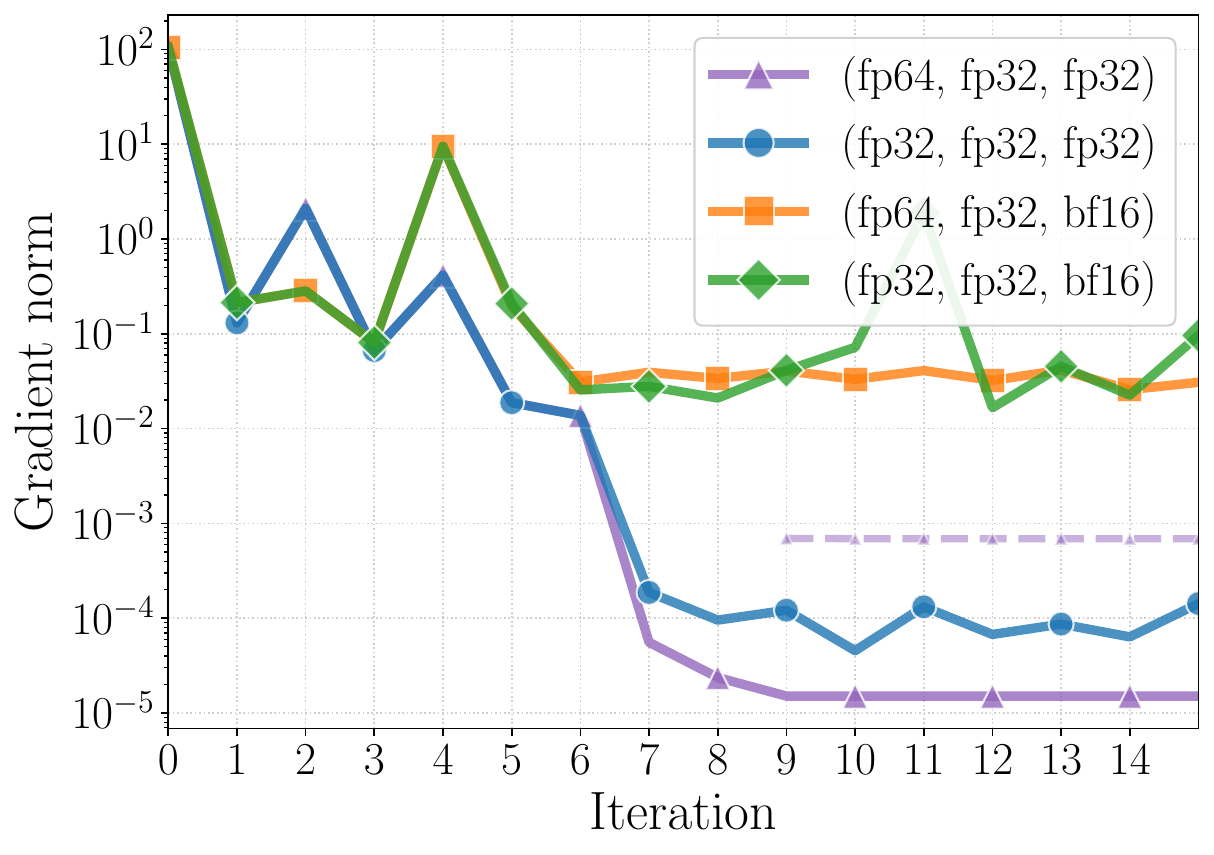}
    \end{minipage}
    \caption{Convergence of mixed precision Newton in relative error (left) and
    gradient norm (right), for different precision sets with unit roundoffs
    $(u_g,u,u_H)$, on the \texttt{SINREG} problem with $n=8$ and $x_0 =
    \bar{x}^* + 10^{-1}$.}
\label{fig:3_err}
\end{figure}

This behavior changes when considering different problems where the Hessian is
ill conditioned or the gradient is strongly approximated, as we will see in subsequent examples.
For instance, setting $n = 8$ in~\zcref{eq:d_dim_model} and taking the new
solution point to be the concatenation of two copies of $\bar{x}^*$, we now have
$\kappa(\Hstar) \approx 4 \times 10^5$. \zcref{fig:3_err} shows the 
convergence of Newton's method on this new problem with the same four precisions sets as previously.

In this setting, the Hessian is so ill-conditioned that the assumptions of \zcref{th:rel_err,th:grad_norm} may not be satisfied if the precisions are too low or if
the starting point is too far from the solution. We only plot the theoretical
bounds (transparent curves) when these assumptions are satisfied. For example, when
$u_H\equiv$ bfloat16 (orange and green curves), the assumptions are never satisfied and in fact
the method diverges. When $u_H\equiv$ fp32 (purple and blue curves), the method does converge,
although the assumptions are only satisfied for the purple curve (when $u_g\equiv$ fp64)%
\footnote{When $u_g \equiv$ fp32 (blue curve), the method stops slightly further from the solution, 
due to $\gamma_i$ being larger, and thus the assumptions are not satisfied and the bounds are not plotted.}
and only starting at iteration $i=8$,
when the current iterate is sufficiently near the exact solution. 
This shows that in such extreme cases the theory can be too conservative and unable to guarantee convergence
even though it is empirically observed.

Moreover, comparing the purple and blue curves reveals the impact of the conditioning on the
relative error, which makes  $\eg$ the dominant error source when $u_g = u$. 
This problem instance thus illustrates that evaluating the gradient in a higher precision than the working precision ($u_g \ll u$)
can be beneficial to improve the limiting accuracy.

\subsubsection{Finite differences} \label{subsec:fd}

In this section we illustrate the effect of the errors coming from the gradient
approximation and their interplay with rounding errors. 
Specifically, we consider forward finite differences
\cite{rando2025fdusage}:
\begin{equation*}
    g(x) \approx \left(\frac{f(x + he_k) - f(x)}{h}\right)_{k=0,\ldots,n-1},
\end{equation*}
where $e_k$ is the $k$th canonical basis vector, 
$\mathbb{R}^+ \ni h  \ll 1$, and the computation is done in precision with unit roundoff $u_g$.
The choice of $h$ is of fundamental importance. 
The error on the gradient coming from forward finite differences is indeed of order $\eg \approx h + u_g/h$~\cite[eq.~(8.5)]{nocedalWright}.
The common approach\footnote{\url{https://nhigham.com/2020/10/06/what-is-the-complex-step-approximation/}}
is thus to choose $h \approx \sqrt{u_g}$, which balances approximation and finite precision errors
and leads to a total error of order $\eg=\sqrt{u_g}$.

\begin{figure*}[ht!]
\centering

\subcaptionbox{$(u_g,u,u_H) \equiv$ (fp64, fp32, fp32), optimal $h\approx10^{-8}$. \label{fig:fd_2_combined}}{%
\begin{minipage}{.9\textwidth}
    \centering
    \begin{minipage}{.5\textwidth}
        \centering
        \includegraphics[width=\linewidth]{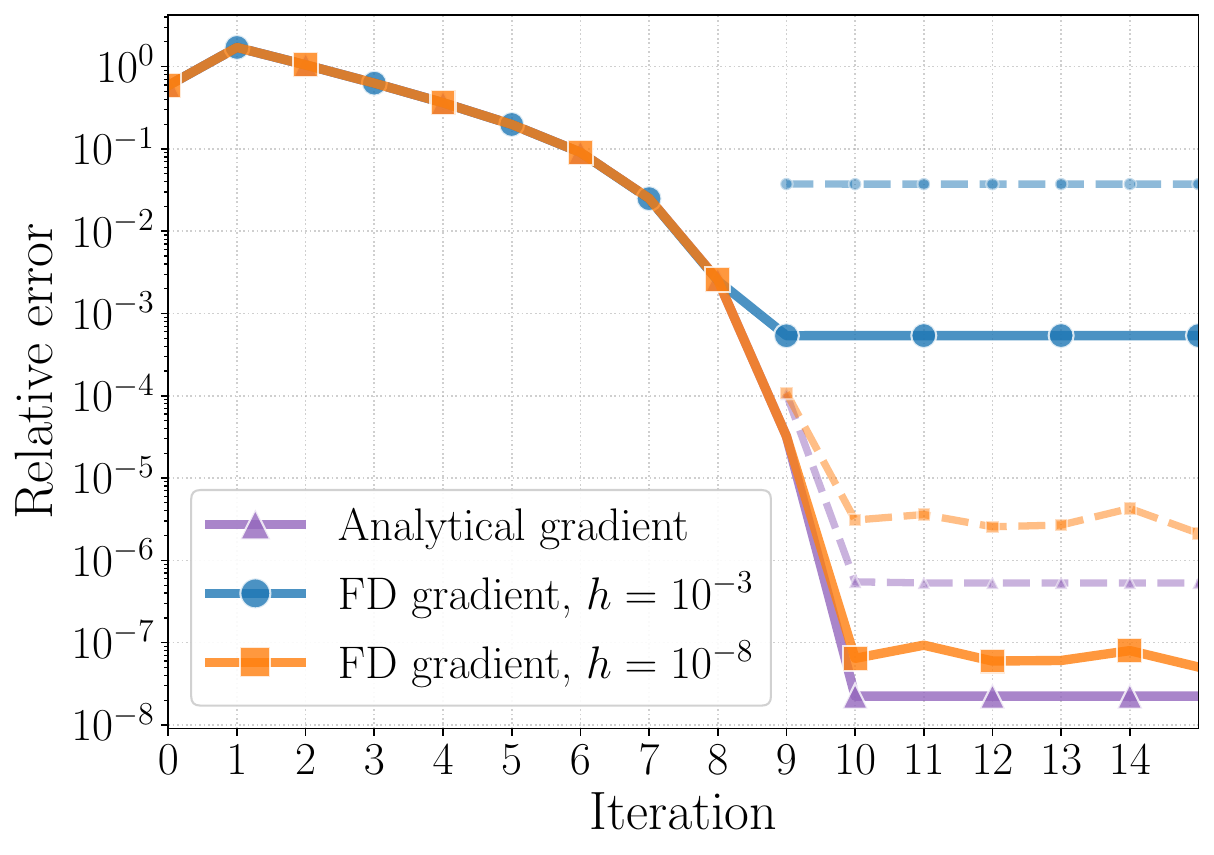}
    \end{minipage}%
    \begin{minipage}{.5\textwidth}
        \centering
        \includegraphics[width=\linewidth]{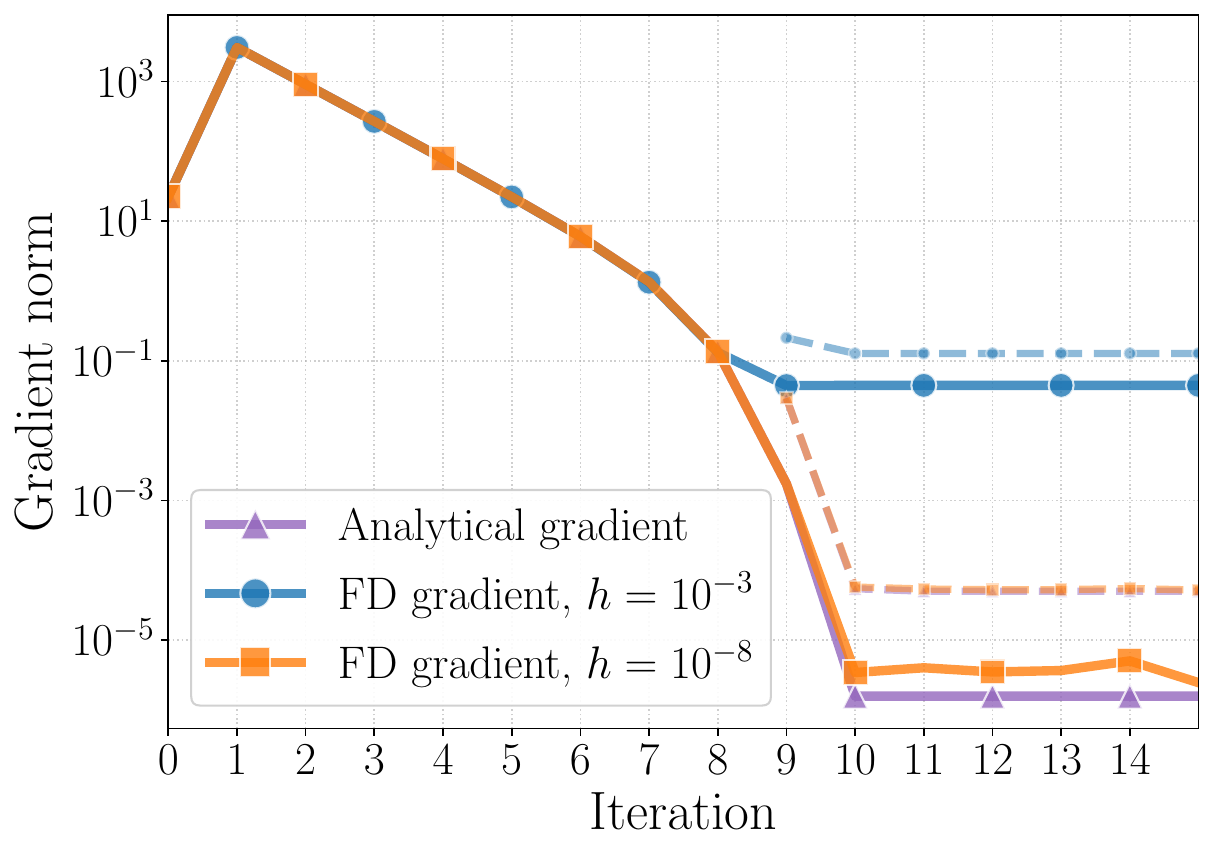}
    \end{minipage}
\end{minipage}
}

\par\vspace{0.5em}

\subcaptionbox{$(u_g,u,u_H) \equiv$ (fp32, fp32, fp32), optimal $h\approx10^{-3}$.\label{fig:fd_3_combined}}{%
\begin{minipage}{.9\textwidth}
    \centering
    \begin{minipage}{.5\textwidth}
        \centering
        \includegraphics[width=\linewidth]{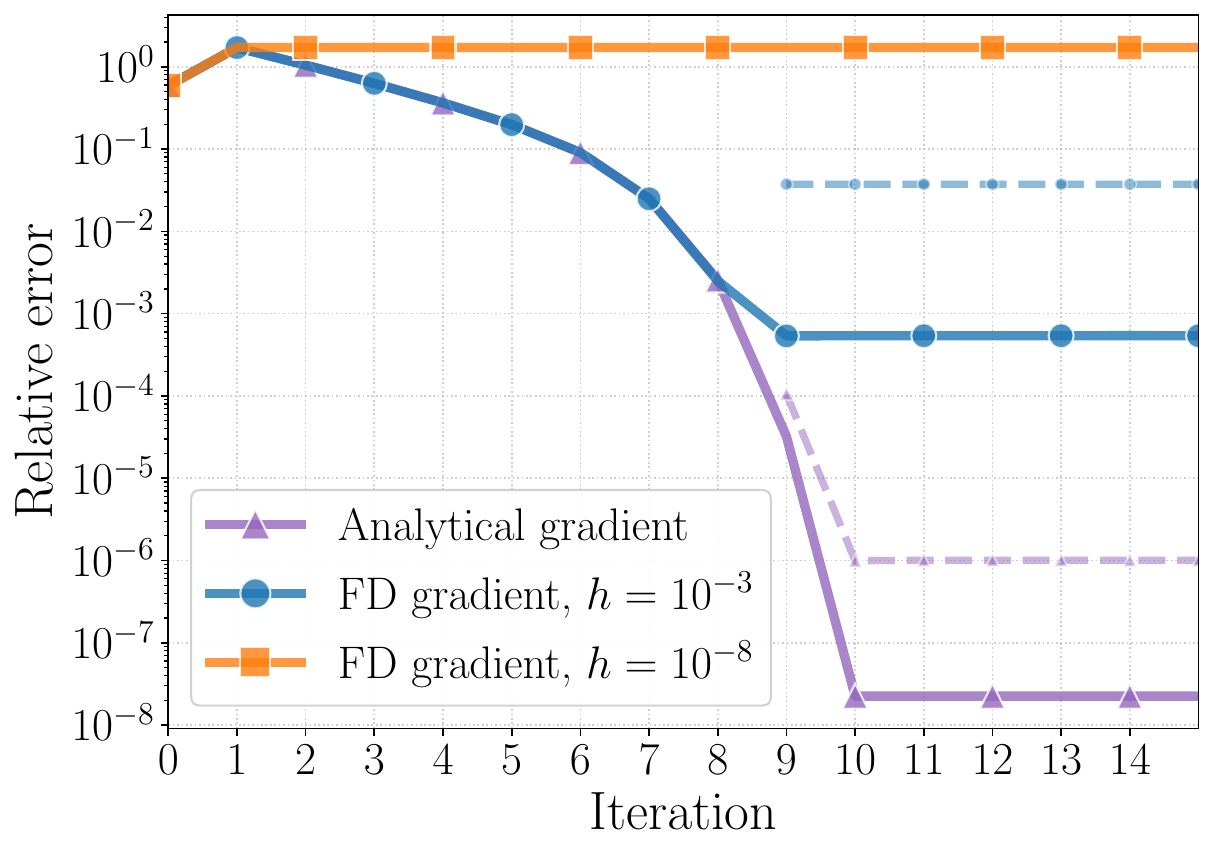}
    \end{minipage}%
    \begin{minipage}{.5\textwidth}
        \centering
        \includegraphics[width=\linewidth]{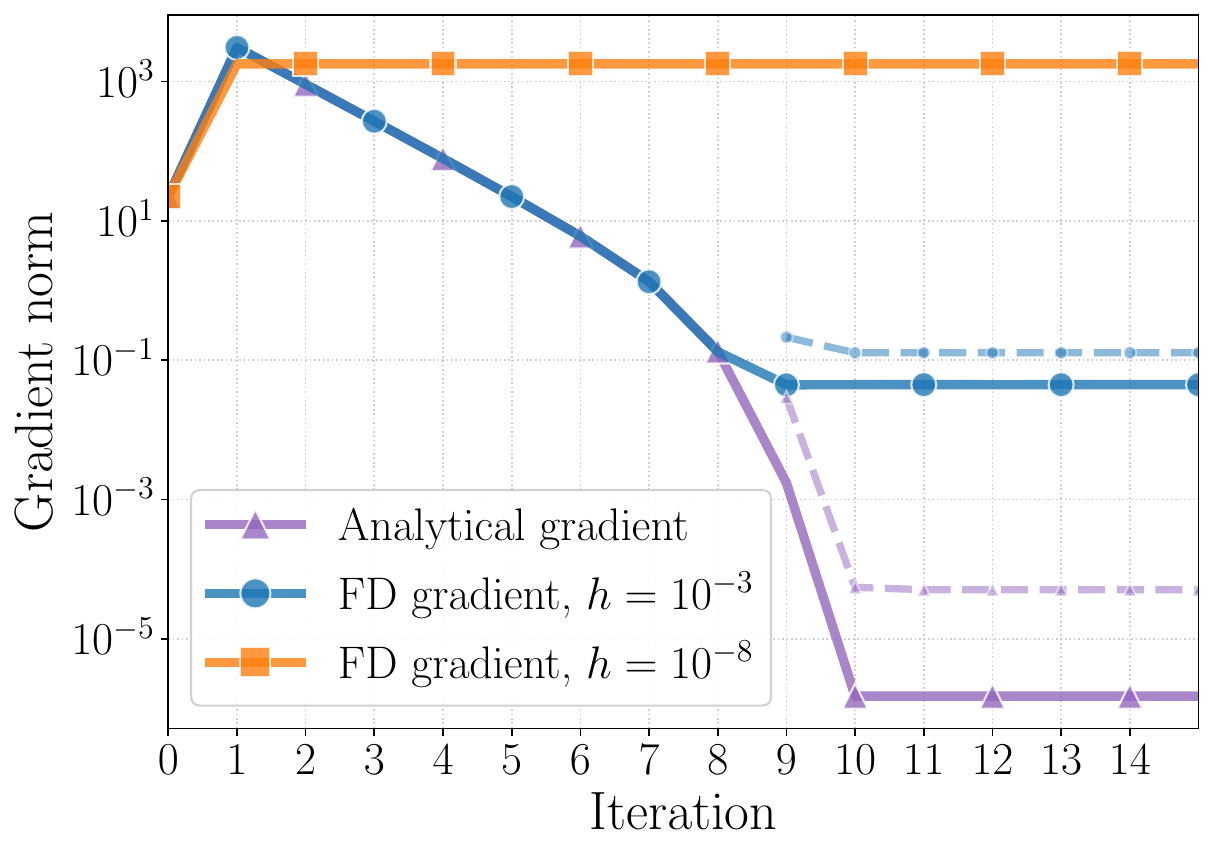}
    \end{minipage}
\end{minipage}
}
\caption{Convergence of mixed precision Newton in relative error (left) and gradient norm (right)
  on the \texttt{ENGVAL1} problem, using either the true gradient or one approximated by finite differences (FD), with different steps $h$ and precision sets.}
\label{fig:all_combined_fd}
\end{figure*}

\zcref{fig:all_combined_fd} compares the convergence of mixed precision
Newton on the \texttt{ENGVAL1} problem (defined in~\zcref{eq:engval1})
with the gradient approximated with finite differences,
for different steps $h$ and for $u_g\equiv$ fp64 or fp32.
We also report the baseline convergence using the exact gradient
to show the impact of these approximations. 

The figure confirms
that the choice $h \approx \sqrt{u_g}$ leads to the best limiting accuracy,
itself of order $\sqrt{u_g}$. Hence, for $u_g = u$ (\zcref{fig:fd_2_combined}), 
the limiting accuracy is significantly worse when
the gradient is approximated than when using the analytical gradient, due to 
$\eg$ dominating the $\limacc$ term. 
The figure also shows that, as expected, this approximation does not impact the
convergence rate of the method. Our framework can thus correctly predict
the behavior of the method even in this case, showing that it can cover
gradient errors that are not solely due to finite precision.  
Finally, this case provides another example where using 
higher precision for evaluating the gradient ($u_g \ll u$)
significantly improves the final accuracy.

\subsubsection{Performance profile on CUTEst problems} \label{subsec:cute}

We now evaluate the behavior of mixed precision Newton on a wider range of test problems
coming from the optimization problems dataset CUTEst~\cite{gratton2024s2mpjcutestoptimizationproblems}. 
We consider 30 problems for which second-order derivatives are available in the python
library, and whose Hessian is not too ill-conditioned ($\kappa(\Hstar) < 10^{15}$); their dimension ranges from 2 to 120.  We compare the uniform fp32 method, 
in which all computations are performed in fp32, with three mixed precision variants
using different precision sets for $(u_g, u, u_H)$:
(fp64, fp32, fp32) uses higher precision for the gradient, (fp32, fp32, bfloat16) uses lower precision for the Hessian,
and (fp64, fp32, bfloat16) does both at the same time, thereby using three different arithmetics.
All the methods use the same working precision, $u\equiv$ fp32.
According to our theory (see~\zcref{def:limg_thm}),
the limiting accuracy on the gradient norm at a given iteration $i$
will be, at best, of order $\norm{\Hi}\norm{\xh_i}u$. 
We therefore use this quantity as stopping criterion:
all the methods run for at most $1000$ iterations and stop earlier 
if the gradient norm becomes smaller than $\norm{\Hi}\norm{\xh_i}u$.

In order to compare the behavior of the different precision sets
on such a large number of problems, we use a performance profile~\cite{dolan2002benchmarking},
displayed in \zcref{fig:CUTE_results}.
The performance profile is built considering the number of iterations as a
performance metric. For each problem, the best performing method is the one
that converges in the least number of iterations, and the others are compared
to it by computing the ratio between their number of iterations and the best
one. The performance profile then reports, for each method, the percentage of
problems for which this ratio is below a certain threshold $\tau$.

\begin{figure}[ht]
\centering
    \centering
    \includegraphics[width=.5\linewidth]{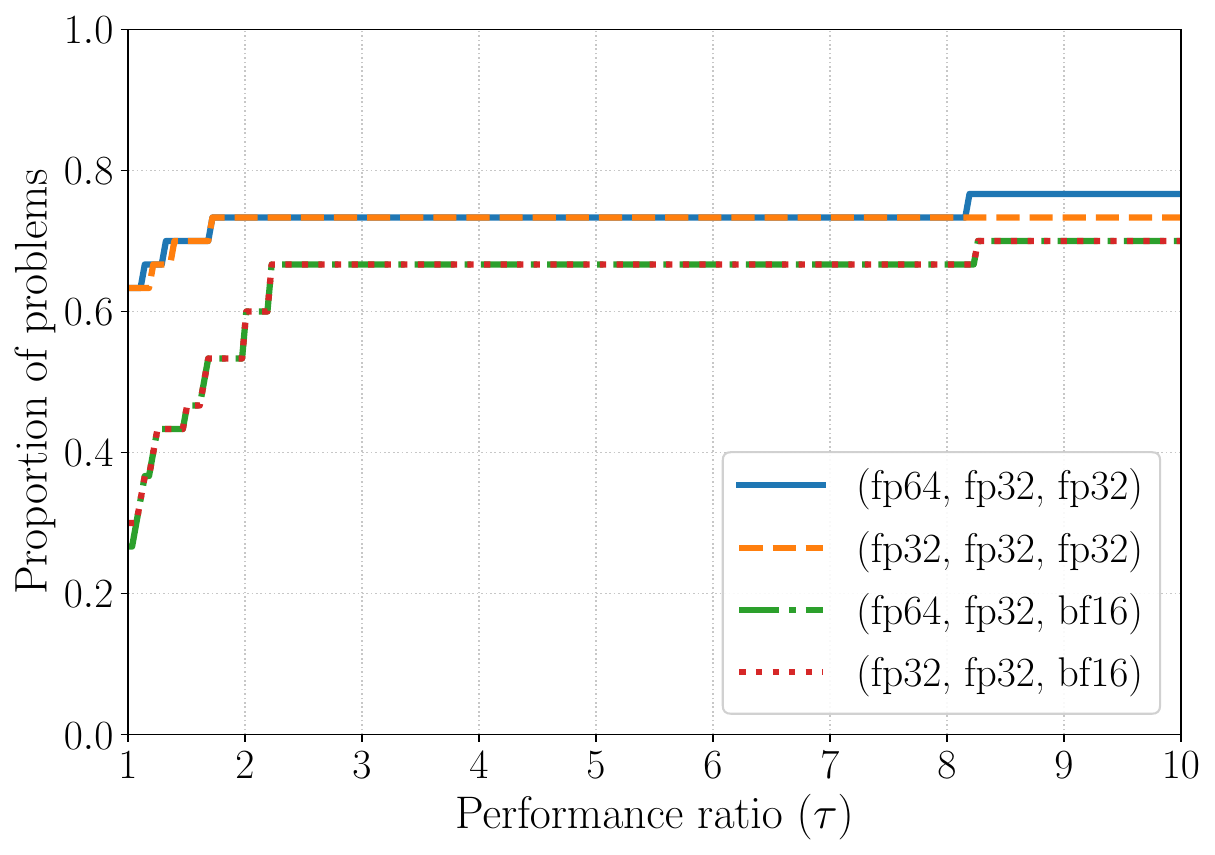}
\caption{Performance profile of uniform and mixed precision Newton's method for
various precision sets on a range of 30 CUTEst problems. The y-axis shows the proportion
of problems successfully solved in less than $\tau$ times the number of iterations
of the best method, where $\tau$ varies on the x-axis.}
\label{fig:CUTE_results}
\end{figure}

The performance profile shows that using a higher precision for the gradient (fp64 instead of fp32) barely
has any impact 
on either the number of successfully solved problems or the number of
iterations. There is only one problem for which using $u_g\equiv$ fp64 instead
of $u_g\equiv$ fp32 allows convergence, namely \texttt{HAIRY}, 
whose Hessian is not ill-conditioned ($\kappa(\Hstar) \approx 10^3$), but which shows a high error on the gradient when using lower precision ($\eg \approx 0.1$ for many iterations when using fp32 everywhere).
This suggests that, at least for these CUTEst problems, the evaluation of gradient in floating-point arithmetic
is typically well behaved and does not require the use of a precision higher than the working precision.

On the other hand, the use of lower precision for the Hessian (bfloat16 instead of fp32) does have a more
visible impact. The number of solved problems is reduced from 23 to 21:
the two problems for which setting $u_H\equiv$ bfloat16 prevents convergence are
\texttt{MEXHAT} and \texttt{PENALTY1}.
For these problems,
the assumption on the Hessian conditioning of~\zcref{th:rel_err} is in fact not
satisfied, which shows that our theory correctly detects the possible lack of convergence\footnote{That
being said, there are other problems, such as \texttt{BROWNAL}, for which our theory is also unable to
guarantee convergence with $u_H\equiv$ bfloat16, but which do converge. This once more illustrates 
that mixed precision Newton can be more robust than what the theory predicts.}.
Moreover, the average number of iterations is increased from about $10$ to about $18$;
this increase is limited to at most a factor $\tau=2$ for 60\% of the problems.
This suggests that many of these CUTEst problems present a sufficiently well-conditioned
Hessian to be solved at a reduced computational cost by using lower precision than
the working precision for the Hessian.

\subsection{Mixed precision inexact Newton}
\label{subsec:in_exp}

We now consider the mixed precision inexact Newton
method proposed in~\zcref{sec:inexact_newton}. We use the CG method
to approximately solve the Hessian linear systems. We 
use \zcref{stopping} as stopping criterion, with a fixed tolerance $\eta_i=\eta$ across all nonlinear iterations $i$. 
The CG solver is initialized with the zero vector, and the maximum number of linear iterations is set to $100$.
We consider the \texttt{ENGVAL1} problem defined in~\zcref{eq:engval1}, which leads to Hessian matrices
with moderate condition numbers between $10^2$ and $10^3$ depending on the nonlinear iteration.

Recall that, according to our theory and as discussed in~\zcref{sec:inexact_newton},
we expect the convergence rate of the method to be mainly driven by
$\epsilon_i^H\approx u_H+\eta/\ki$, with $\ki$ defined in~\zcref{def:ki}.
Moreover, for this problem, we have $\ki \approx 1$ for all $i$.
Hence, we expect the convergence rate to be determined by the maximum of $u_H$ and $\eta$.

\begin{figure}[ht]
    \centering

    \begin{subfigure}{.48\textwidth}
        \centering
        \includegraphics[width=\linewidth]{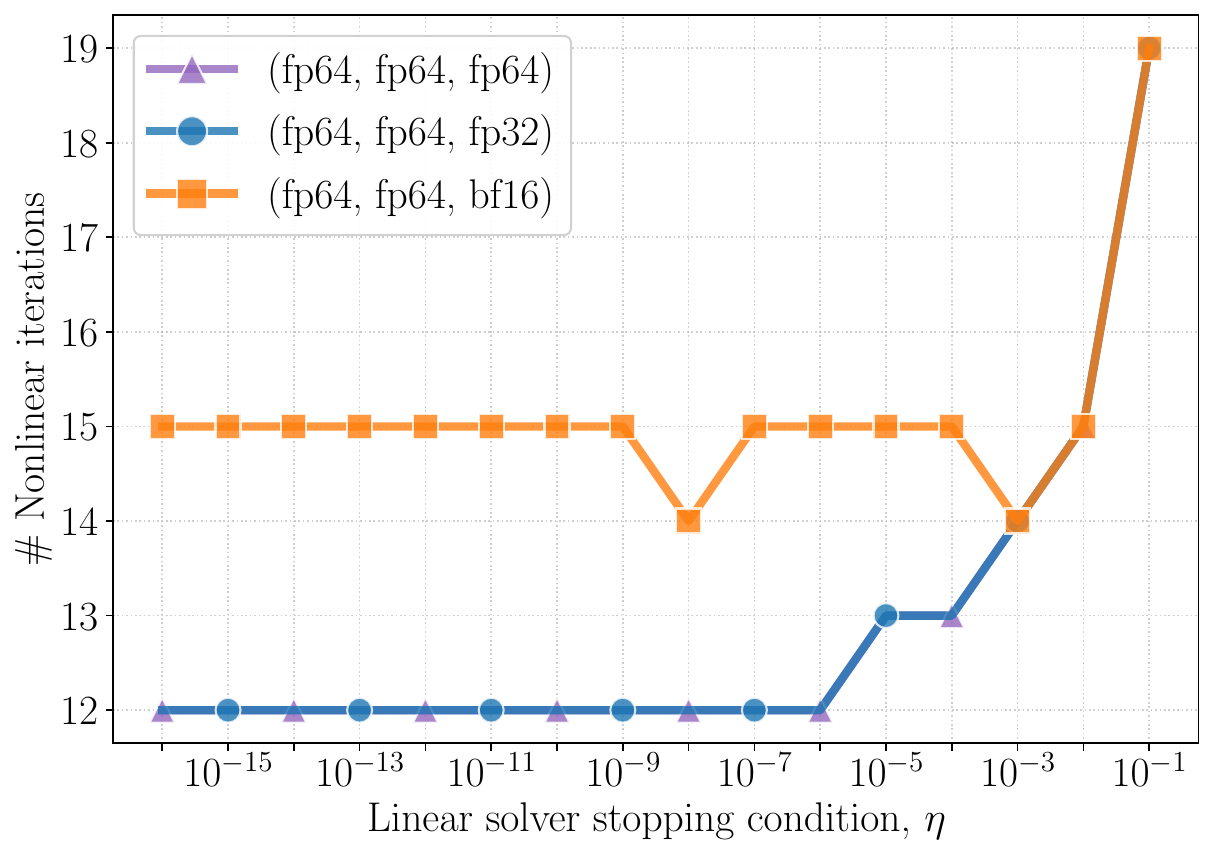}
        \caption{Nonlinear its. vs CG tolerance $\eta$.}
        \label{fig:in_mult_prec_eta}
    \end{subfigure}
    \hfill
    \begin{subfigure}{.48\textwidth}
        \centering
        \includegraphics[width=\linewidth]{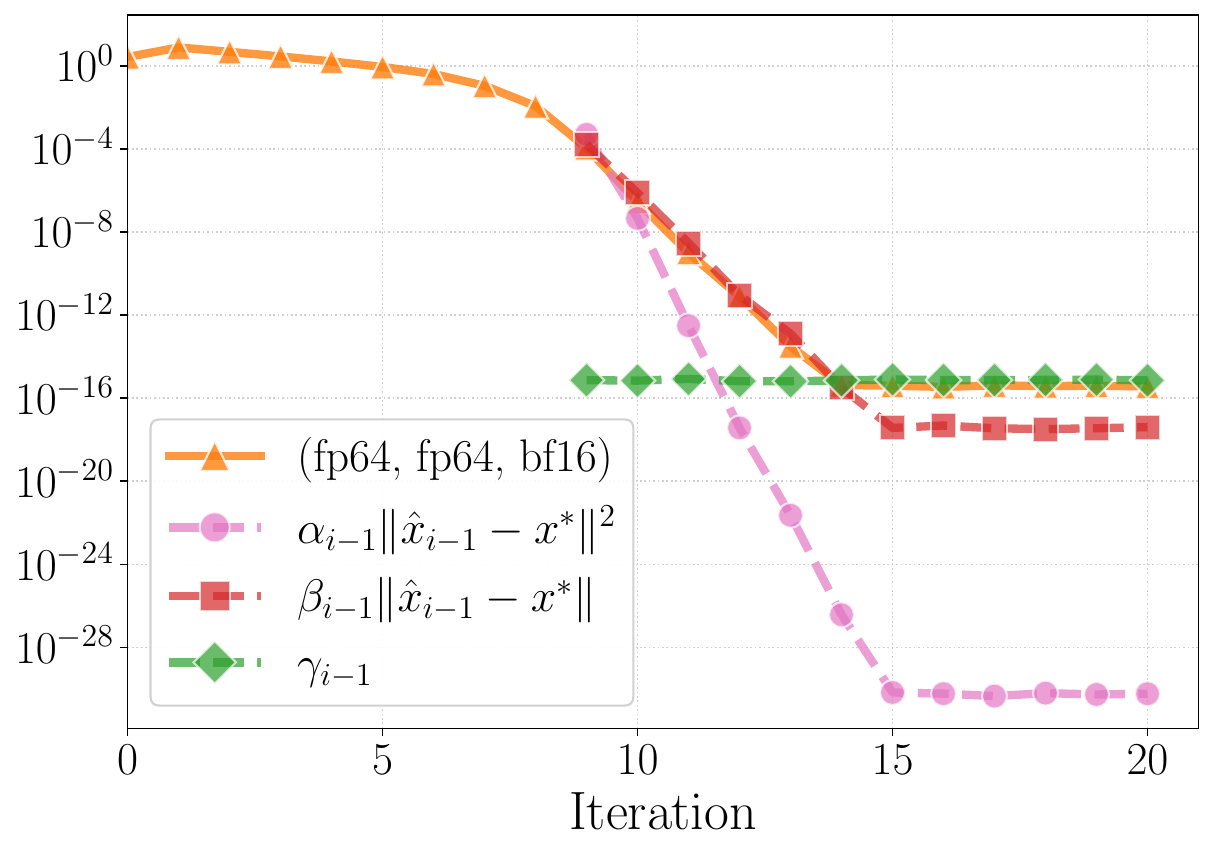}
        \caption{$(u_g,u,u_H) \equiv (\text{fp64},\text{fp64},\text{bf16})$.}
        \label{fig:conv_b16}
    \end{subfigure}

    \vspace{0.5em}

    \begin{subfigure}{.48\textwidth}
        \centering
        \includegraphics[width=\linewidth]{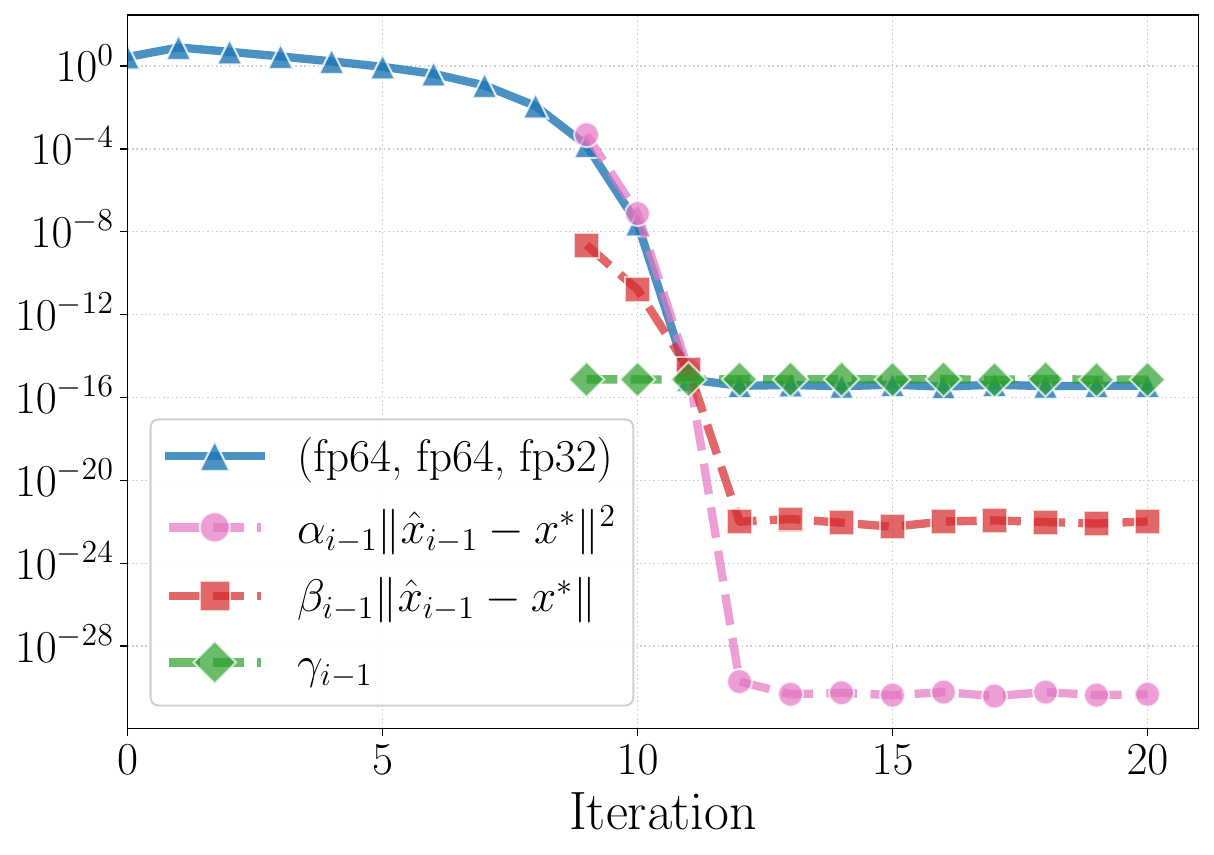}
        \caption{$(u_g,u,u_H) \equiv (\text{fp64},\text{fp64},\text{fp32})$.}
        \label{fig:conv_fp32}
    \end{subfigure}
    \hfill
    \begin{subfigure}{.48\textwidth}
        \centering
        \includegraphics[width=\linewidth]{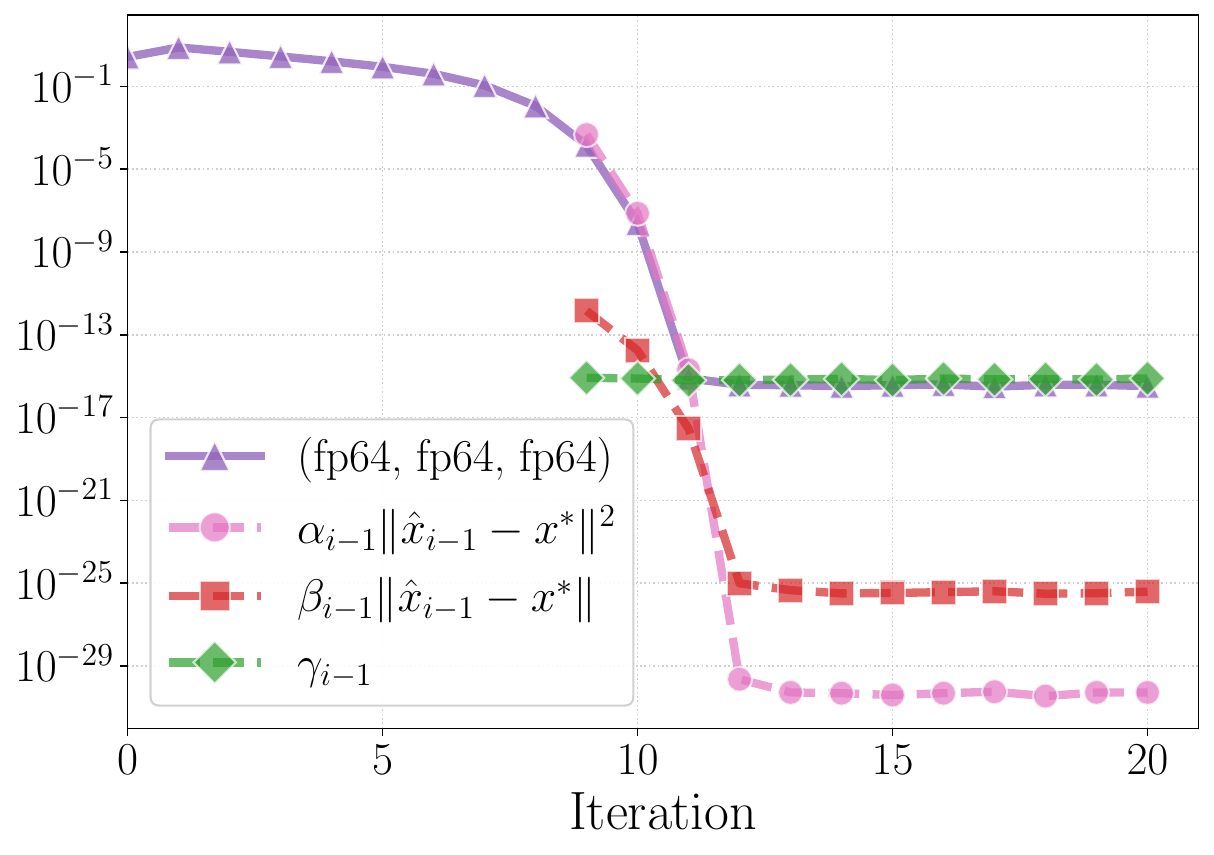}
        \caption{$(u_g,u,u_H) \equiv (\text{fp64},\text{fp64},\text{fp64})$.}
        \label{fig:conv_fp64}
    \end{subfigure}

    \caption{
    Mixed precision inexact Newton on the \texttt{ENGVAL1} problem.
    Panel~(a) reports the number of nonlinear iterations required for convergence
    as a function of the CG stopping tolerance $\eta$ for several precision sets
    $(u_g,u,u_H)$; the purple and blue curves overlap.
    Panels~(b)--(d) show the convergence in relative error for $\eta=10^{-10}$
    for three representative precision sets. The three terms of the theoretical
    bound are also displayed separately whenever the assumptions of
    \zcref{th:rel_err} are satisfied.
    }
    \label{fig:in_mult_prec_combined}
\end{figure}

We confirm this experimentally in \zcref{fig:in_mult_prec_eta}.
We compare the number of nonlinear iterations required by mixed precision inexact Newton's method,
with precisions $u_g=u\equiv$ fp64 and varying $u_H$ and $\eta$.
The figure shows that, as long as the stopping tolerance is large ($\eta \ge10^{-3}$),
there is no difference between using $u_H\equiv$ fp64, fp32, or bfloat16,
because the error is dominated by the solver's inexactness. Thus, for such highly approximate solvers,
we may safely use low precision arithmetic for the Hessian without impacting the convergence rate.
For smaller tolerances $\eta \le 10^{-4}$, a difference between $u_H\equiv$ bfloat16 and
$u_H\equiv$ fp32 or fp64 appears, with the former stagnating at 14--15 nonlinear iterations regardless of
$\eta$, since the error is then dominated by $u_H$.

One could expect to see a similar difference appear between $u_H\equiv$ fp64 and $u_H\equiv$ fp32 
when $\eta$ becomes smaller than the fp32 unit roundoff, but this is not the case.
The reason for this behavior is explained in~\zcref{fig:conv_fp64,fig:conv_fp32}, which compare the relative error
convergence with $u_H\equiv$ fp64 and $u_H\equiv$ fp32, and plots the three terms composing its theoretical bound separately.
It shows that the dominant term is first $\alpha_i\norm{\xh_i-x^*}^2$, which decreases
with $i$, until it becomes smaller than $\beta_i\norm{\xh_i-x^*}$; however, at that point,
the relative error is already below the limiting accuracy $\limacc$. Hence, the unit roundoff $u_H$, which only appears in
the term $\beta_i$, does not have any impact on the convergence rate in this situation. On the other hand, looking at the convergence with $u_H\equiv$ bfloat16 in~\zcref{fig:conv_b16}, we see that the term $\beta_i\norm{\xh_i-x^*}$ becomes dominant earlier, and so delays convergence by a few iterations.

Finally, we want to validate the theoretical bounds 
when $\ki \gg 1$ for some $i$, in which case $\eta_i$ should not be fixed across iterations, but
rather depend on the current iterate through $\ki$. 
In this experiment, we use the same test case as in~\zcref{fig:3_err},
for which $\ki$ ranges from $10^1$ to $10^3$, with no particular pattern across
iterations. In this setting, our bounds predict that the stopping condition
derived in~\zcref{sec:inexact_newton} ($\eta_i = \ki u_H$) should be as effective
as $\eta_i = u_H$, that is, should guarantee the same convergence rate as if using a
direct solver, while being less computationally expensive. 

\begin{figure}[ht]
    \centering
        \begin{minipage}{.5\textwidth}
            \centering
            \includegraphics[width=.95\linewidth]{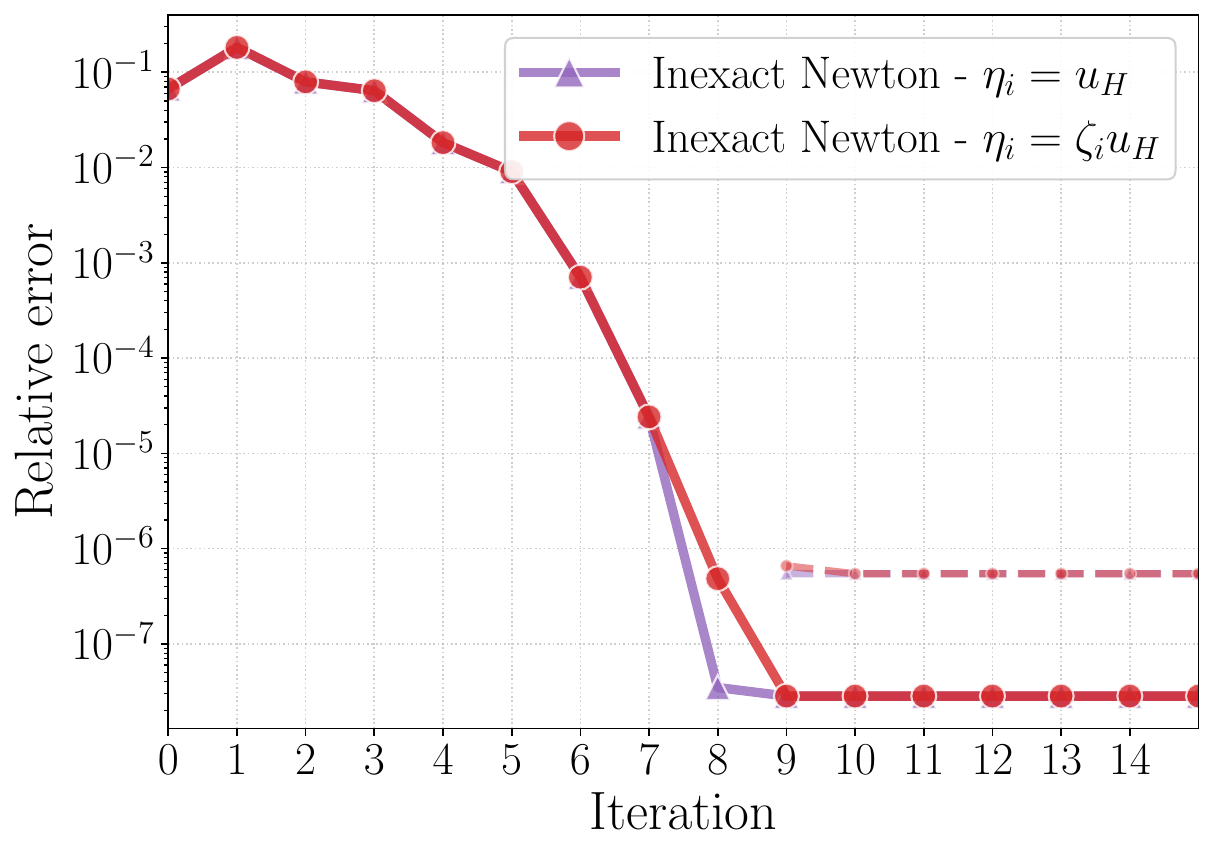}
        \end{minipage}%
        \begin{minipage}{.5\textwidth}
            \centering
            \includegraphics[width=.95\linewidth]{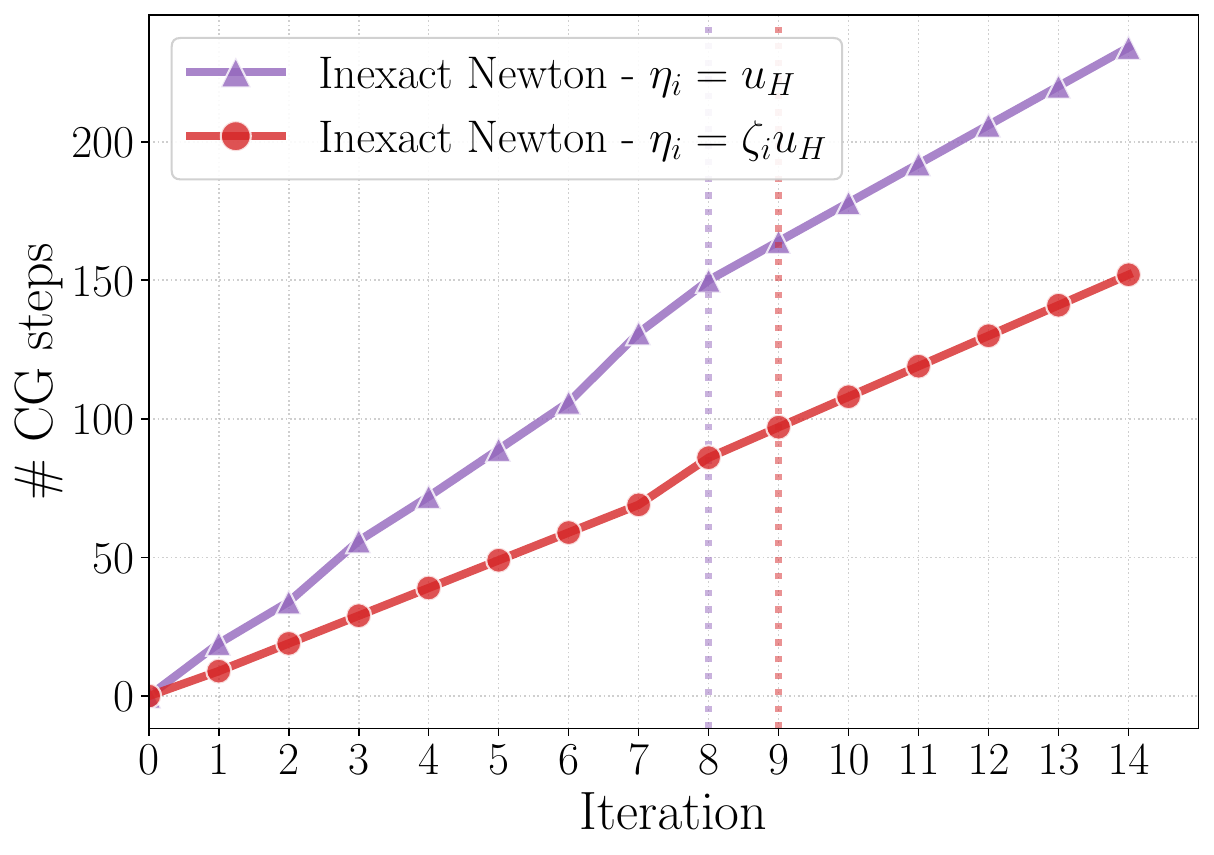}
        \end{minipage}
    \caption{Relative error (left) and number of linear CG iterations (right)
for mixed precision inexact Newton, using precisions (fp64, fp32, fp32) with
different CG stopping conditions. The problem is \texttt{SINREG1} with $n =
8$ and $x_0 = \bar{x}^* + 10^{-1}$. The vertical dotted lines in the right plot indicate the
iteration at which the corresponding algorithm reached the limiting accuracy.}
    \label{fig:cg_stop_cond}
\end{figure}

In~\zcref{fig:cg_stop_cond} we confirm that, for this problem, setting $\eta_i
= u_H$ for all $i$ guarantees the same converge rate as using a direct solver.
This is visible by comparing the purple curve with the one
of~\zcref{fig:3_err}. Moreover, we see that setting $\eta_i = \ki u_H$ slightly
increases the number of nonlinear iterations needed for the method to converge,
but significantly reduces the number of CG steps overall, from 150 (across 8 nonlinear iterations)
to 100 (across 9 nonlinear iterations).

\subsection{Mixed precision Gauss--Newton}
\label{subsec:gn_exp}

We conclude our experiments with the mixed precision Gauss--Newton method proposed in~\zcref{sec:gauss_newton}.
We consider the least-squares problem \texttt{SINREG} defined
in~\zcref{eq:regression_base}--\zcref{eq:d_dim_model}, again using $n=4$ and
$\bar{x}^* = (1.0123, 2.01234, 1.01231, 2.01234)$, and 
with starting point $x_0 = \bar{x}^* + 10^{-1}$.
We solve the Gauss--Newton system by LU factorization in precision with unit roundoff $u_H$.
Thus, in this case, we expect the error term $\eh$ to be driven by the maximum of $u_H$ and
the relative norm of the discarded term $S(\xh_i)$.
We recall that the noise $\xi$ added to the model outputs $y_i$ is sampled
from a uniform $[-\delta, \delta]$ distribution.  To make the norm
of the discarded term $S(\xh_i)$ vary, we test different noise sizes $\delta$.

In~\zcref{fig:gn_vs_n_0}, we compare the convergence in gradient norm of
mixed precision Newton and Gauss--Newton; the relative error behaves similarly. 
We set $u=u_g\equiv$ fp64 and we compare $u_H\equiv$ fp64 (left plot) with $u_H\equiv$ bfloat16 (right plot). 
The noise size is here set to $\delta = 5$, which leads to a discarded term of quite large
relative norm $\norm{S(\xh_i)}/\norm{H(\xh_i)}$ (starting around $3\times
10^{-2}$ for the first iterations and decreasing to $3\times 10^{-4}$ at convergence).
As a result, in the left plot (with $u_H\equiv$ fp64), the Hessian error $\eh$ is dominated by this term and is much larger
with Gauss--Newton than with standard Newton, so the former converges at a significantly slower rate than the latter.
On the other hand, in the right plot (with $u_H\equiv$ bfloat16), the Hessian error is now dominated
by the rounding errors in bfloat16 arithmetic, and the two methods achieve a similar convergence rate.
This illustrates that lower precision arithmetic can be safely used for the Hessian when
it is already approximated by Gauss--Newton or, conversely, that Gauss--Newton can safely replace
standard Newton if the Hessian precision is low.

\begin{figure}[ht]
\centering
    \begin{minipage}{.5\textwidth}
        \centering
        \includegraphics[width=.95\linewidth]{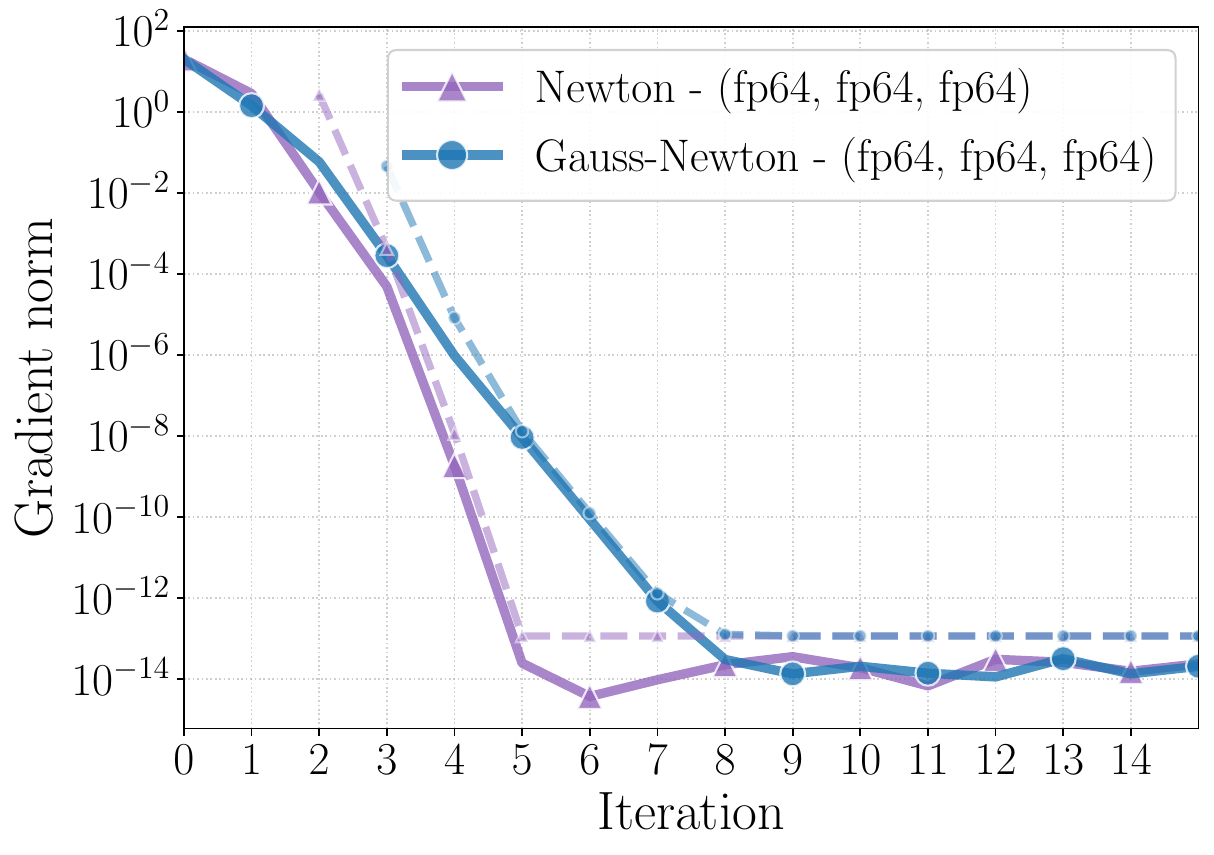}
    \end{minipage}%
    \begin{minipage}{.5\textwidth}
        \centering
        \includegraphics[width=.95\linewidth]{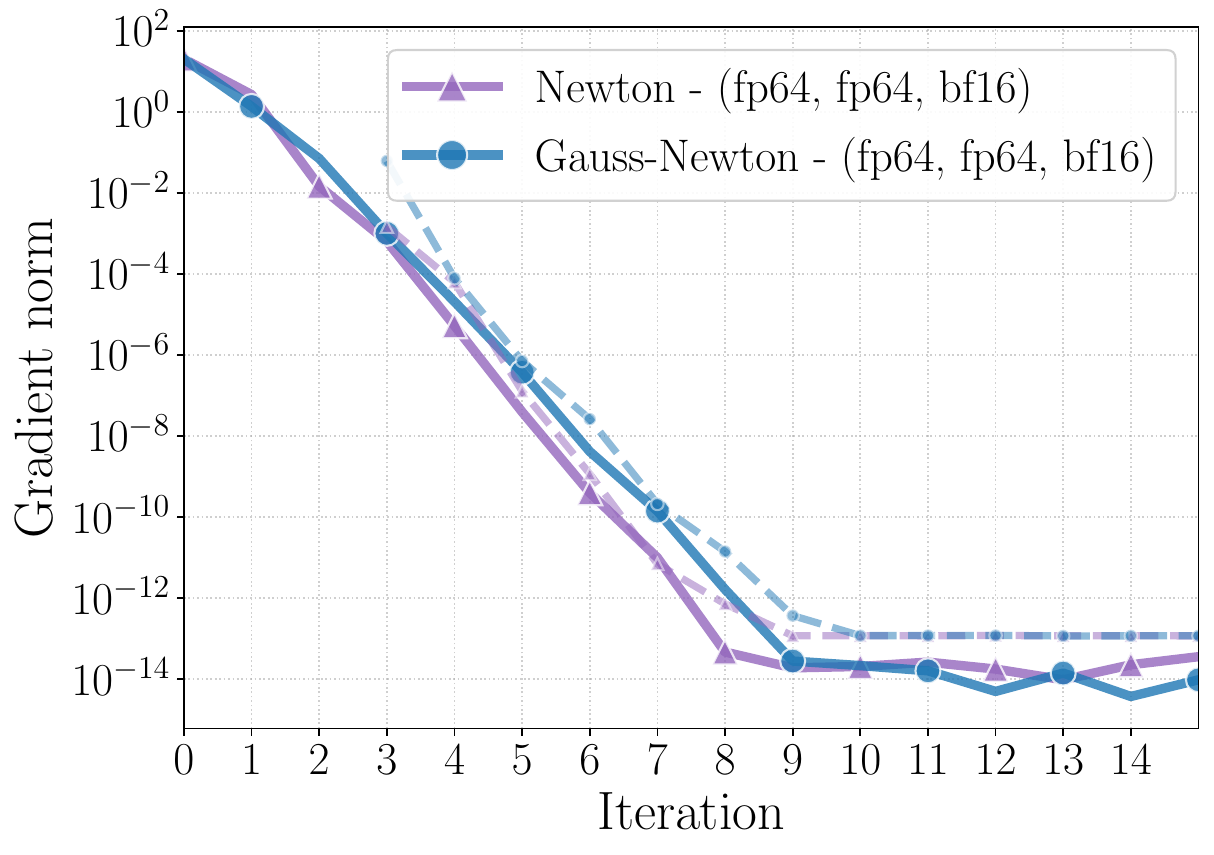}
    \end{minipage}
    \caption{Convergence of mixed precision Newton and Gauss--Newton in
    gradient norm for different precision sets with unit roundoffs
    $(u_g,u,u_H)$, on the \texttt{SINREG} problem with $n=4$, $x_0 = \bar{x}^*
    + 10^{-1}$, and $\delta = 5$.}
    \label{fig:gn_vs_n_0}
\end{figure}

\begin{figure}[ht]
\centering
    \begin{minipage}{.5\textwidth}
        \centering
        \includegraphics[width=.95\linewidth]{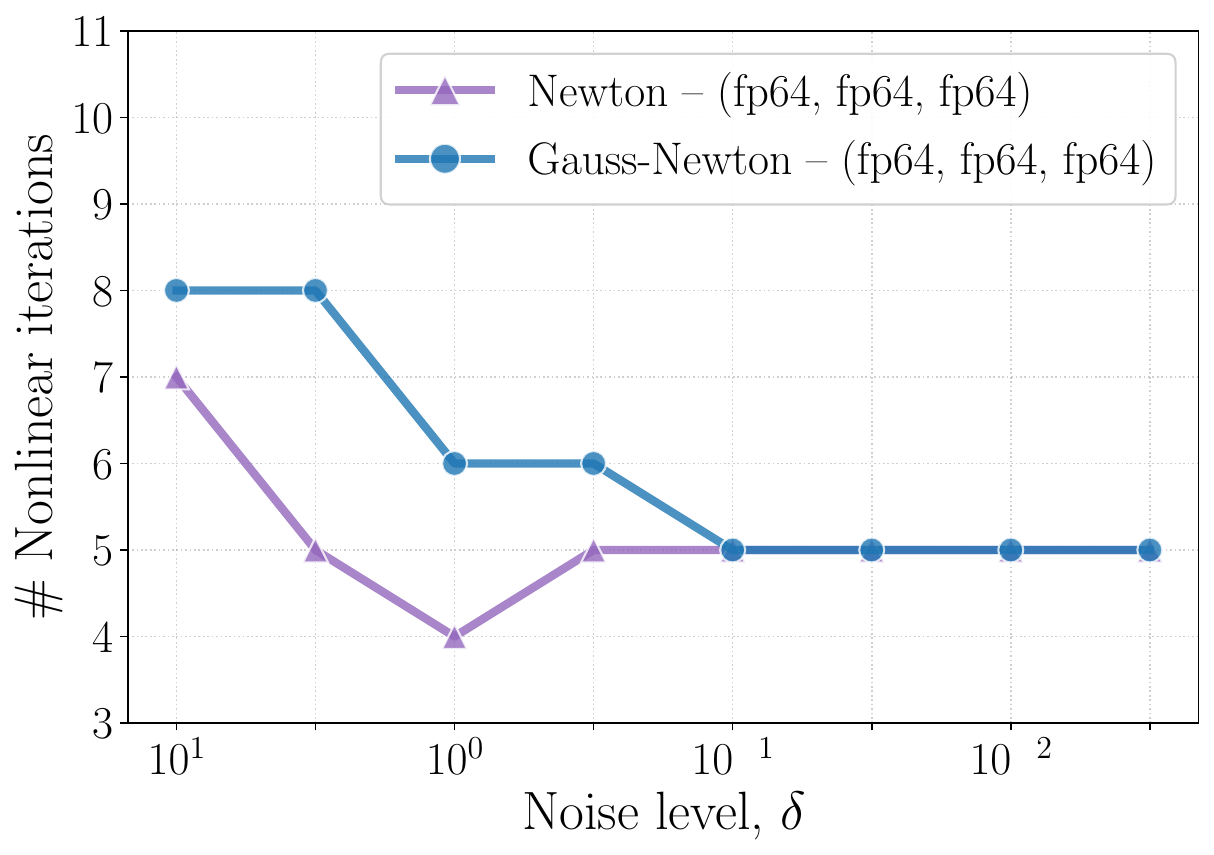}
    \end{minipage}%
    \begin{minipage}{.5\textwidth}
        \centering
        \includegraphics[width=.95\linewidth]{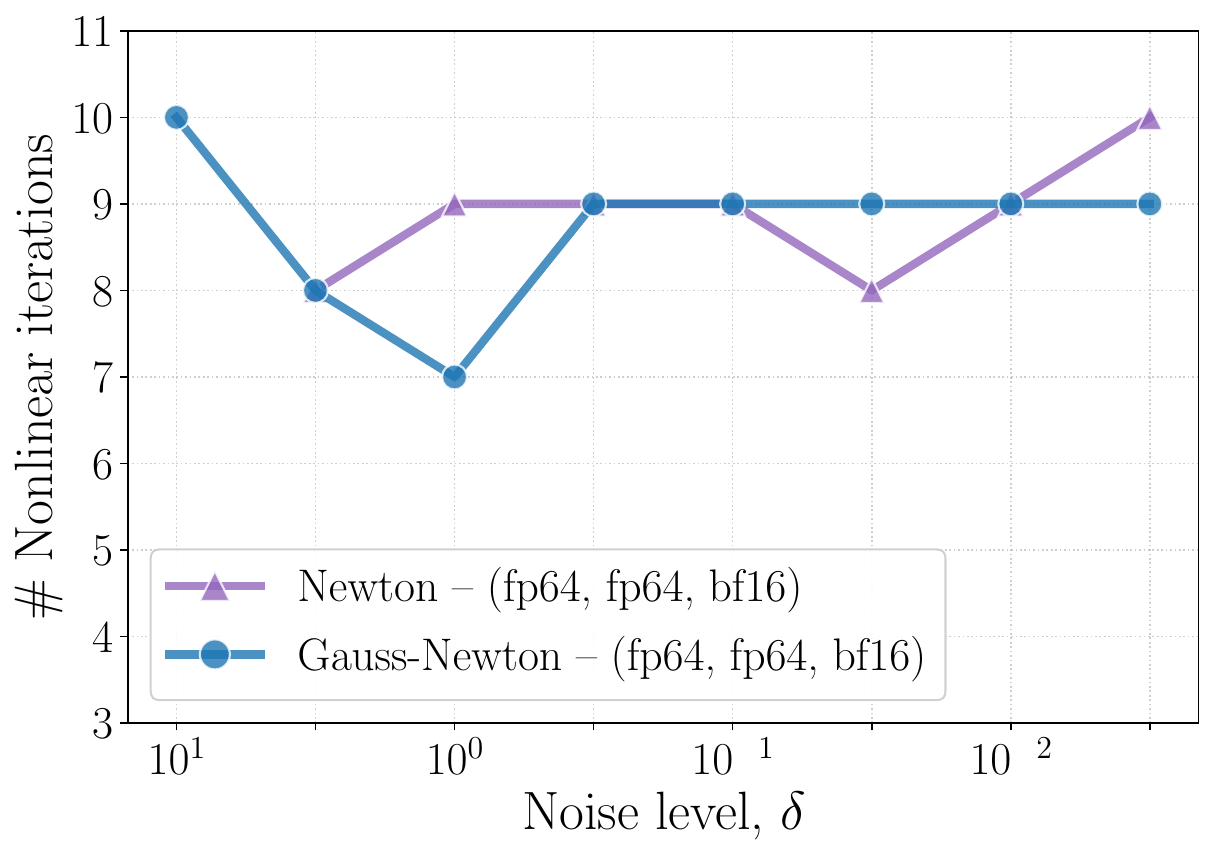}
    \end{minipage}
    \caption{Iterations needed to converge to limiting gradient for mixed precision Newton and Gauss--Newton, for different precision sets with unit roundoffs
    $(u_g,u,u_H)$, on the \texttt{SINREG} problem with $n=4$, $x_0 = \bar{x}^* + 10^{-1}$, and varying noise size $\delta$.}
    \label{fig:gn_vs_n_noise}
\end{figure}

In~\zcref{fig:gn_vs_n_noise}, we plot the number of nonlinear iterations
required by the two methods for different noise sizes $\delta$. When $\delta$
is small, there is no significant difference between the two methods, even with
$u_H\equiv$ fp64, because the residual term discarded when using Gauss--Newton
is small and does not dominate the Hessian error. As $\delta$ increases, when
using $u_H\equiv$ fp64, the convergence of Gauss--Newton becomes
slower than that of Newton. On the other hand, when using 
$u_H\equiv$ bfloat16, the two methods converge in a similar number of iterations even for large 
$\delta$; interestingly, Gauss--Newton may sometimes converge slightly faster than Newton.

\section{Conclusion}\label{sec:conclusions}

We have presented a general framework for mixed precision Newton's method for
optimization, where the three main operations of the algorithm (gradient
evaluation, Hessian system, and solution update) are affected by inexactness. 
Our error model is generic and can be applied to various sources of inexactness, including
floating-point arithmetic, approximate linear solvers (inexact Newton), Hessian matrix approximations
(Gauss--Newton), and any combination thereof. The main results of our analysis are 
\zcref{th:rel_err,th:grad_norm},
which show how the convergence rate and attainable accuracy of Newton's method are affected by
these different sources of inexactness and provide guidelines for choosing
the precisions of the different operations. 
For the inexact Newton and Gauss--Newton variants, 
we link our theory with known convergence results for exact
arithmetic in the literature and discuss what changes in
floating-point arithmetic, highlighting the interplay between these approximate
variants and rounding errors.
We have performed an extensive set of numerical experiments to validate the theoretical analysis and
illustrate the interesting behavior of mixed precision Newton. Our results show
that the empirically observed convergence rate and attainable accuracy match their theoretically 
predicted behavior when the assumptions underlying our theory are satisfied, and that the bounds are quite descriptive.
Tests on a broad range of problems from the CUTEst dataset highlight the robustness and wide applicability of mixed precision Newton's method,
and suggest a significant potential for using lower precisions while only marginally sacrifing the accuracy and convergence rate.

This work opens the way to many promising perspectives:
\begin{itemize}
\item the error model and convergence theory could be extended to other
quasi-Newton methods, such as BFGS~\cite[chap.~6.1]{nocedalWright}, where Hessian inexactness depends on the
gradient accuracy;
\item the analysis could be extended to stochastic optimization~\cite{byrd2015stochasticquasinewtonmethodlargescale},
where inexactness in gradient and Hessian evaluations may also arise from data sampling;
\item mixed precision Newton could be compared with perturbed first-order
methods, such as the gradient descent method analysed
in~\cite{vasin2026lowerupperboundsconvergence}, to assess whether the use of
Hessian information continues to provide advantages when operating under low or
mixed precision arithmetic;
\item a practical high-performance implementation of mixed precision Newton and its variants
could be developed, to assess quantitatively the performance gains in memory, time, and energy
that can be achieved by the use of lower precisions on modern hardware.
\end{itemize}

\section*{Acknowledgments}

Funding for the PhD thesis of G. C. was provided by the Graduate+ MATHINFI Programme. 
This work was also partially supported by the Fondation Simone et Cino Del Duca and by projects
managed by the French National Research Agency (ANR):
France 2030 NumPEx Exa-MA (ANR-22-EXNU-0002),
PEPR IA SHARP (ANR-23-PEIA-0008),
MixHPC (ANR-23-CE46-0005-01),
FPT-4 (ANR-24-CE46-7572), and MEPHISTO (ANR-24-CE23-7039). 

\bibliographystyle{plainurl}
\bibliography{biblio}

\appendix

\section{Proofs}

\subsection{Proof of~\zcref{th:rel_err}} \label{sec:proof_rel_err}

For the proof of this theorem, we will need the following two lemmas. 

\begin{lemma}[Lem.~4.1.12, \cite{dennis_schnabel}]\label{lemma21}
Let $g: \bbR^n \to \bbR^m$ be continuously differentiable in the open convex set $\Omega \subseteq \bbR^n$. Assume, for $v \in \Omega$, the matrix $H$ of the first-order derivatives of $g$ to be $L_H$-Lipschitz continuous at $v$ in the neighborhood $\Omega$. Then, for any $w \in \Omega$, it holds
\[
\norm{ g(w) - g(v) - H(v)(w-v)} \leq \frac{L_H}{2}\norm{ w- v} ^2. 
\]
\end{lemma}

\begin{lemma}[Thm.~3.1.4, \cite{dennis_schnabel}]\label{lemma22}
Let $M \in \bbR^{n \times n}$. If $\norm{M} < 1$, then $(I+M)^{-1}$ exists and 
\[
  \norm{(I + M)^{-1}} \leq \frac{1}{1 - \norm{M}}.
\]
Also, let $A, B \in \bbR^{n \times n}$ and $\norm{A^{-1}(B-A)} < 1$, then $B$ is nonsingular and
\[
  \norm{B^{-1}} \leq \frac{\norm{A^{-1}}}{1 - \norm{A^{-1}(B-A)}}.
\]

\end{lemma}

\begin{proof}[Proof of \zcref{th:rel_err}]
    Since $H$ is continuous and $\Hstar$ is positive definite,  there exists
    $\rho > 0$ such that $\forall x \in B_{\rho}(x^*)$, $H(x)$ is positive
    definite, and thus invertible.

    First we prove that if $\xh_i \in B_{\rho}(x^*)$, then $\xh_{i+1}$ is
    well defined, that is, that $\Hi + \errh$ is invertible. To do so, we rewrite the perturbed Hessian:
    \begin{equation*}
        \Hi + \errh = \Hi(I + \Hi^{-1}\errh).
    \end{equation*}
    Using \zcref{lemma22} with $M\equiv \Hi^{-1}\errh$ 
    shows that $I + \Hi^{-1}\errh$ is  invertible since,
    by the definition of $\errh$ in~\zcref{E^H} and by~\zcref{def:nu_i_thm}, we have 
    \begin{equation} \label{der0}
            \norm{ \Hi^{-1}\errh} \leq \norm{ \Hi^{-1}} \eh\norm{ \Hi } 
                        = \eh\kappa(\Hi)= \nu_i
           < 1. 
    \end{equation}
    As the product of two invertible matrices,  $\Hi + \errh$ is then invertible, with inverse given by
    \begin{equation}\label{inv}
    (\Hi+\errh)^{-1} =(I+\Hi^{-1}\errh)^{-1}\Hi^{-1}. 
    \end{equation}
    
    Now, we will bound the error $\norm{\xh_{i+1}-x^*}$. By
    \zcref{eq:single_eq_newton}, we have
 \begin{align*}       
        \xh_{i+1}-x^* &= \xh_i - x^* - (\Hi+\errh)^{-1}(\gi+\errg) + \errsum \\
        &= \big(I - (\Hi+\errh)^{-1}\Hi\big)(\xh_i-x^*) \\
        &\qquad - (\Hi+\errh)^{-1}\big(\gi - \Hi(\xh_i-x^*) + \errg\big) + \errsum.
    \end{align*}
    Taking norms yields
    \begin{multline} \label{bigdis}
        \norm{ \xh_{i+1}-x^* } \leq \norm{ I - (\Hi+\errh)^{-1}\Hi}\norm{ \xh_i-x^*} \\ 
        + \norm{ (\Hi+\errh)^{-1}}\big(\norm{ \gi - \Hi(\xh_i-x^*) } + \norm{ \errg }\big) 
        + \norm{ \errsum }.
    \end{multline}
    We now reformulate the first term of the right-hand side:
    \begin{align*}
      I - (\Hi+\errh)^{-1}\Hi &= (\Hi+\errh)^{-1}(\Hi+\errh-\Hi) \\
      &= (\Hi+\errh)^{-1}\errh \\
      &= (I+\Hi^{-1}\errh)^{-1}\Hi^{-1}\errh,
   \end{align*}
   and using~\zcref{lemma22} and~\zcref{der0} we can bound its norm by
    \begin{equation}\label{firstterm}
      \norm{I - (\Hi+\errh)^{-1}\Hi} \le \frac{\nu_i}{1-\nu_i}. 
    \end{equation}
   Then, using~\zcref{lemma22} again with $M\equiv \Hi^{-1}\errh$ and~\zcref{der0}, we obtain
    \begin{align} 
            \norm{(\Hi+\errh)^{-1}}&=\norm{(I+\Hi^{-1}\errh)^{-1}\Hi^{-1}} \nonumber \\
            &\leq \norm{(I+\Hi^{-1}\errh)^{-1}}\norm{\Hi^{-1}} \nonumber \\
            &\leq \frac{\norm{ \Hi^{-1} }}{1 - \norm{ \Hi^{-1}\errh}} \leq \frac{\norm{ \Hi^{-1}}}{1-\nu_i}. \label{der1.4}
    \end{align}
    Since $g(x^*)=0$, 
    \zcref{lemma21} with $w\equiv \xh_i$ and $v\equiv x^*$ yields
    \begin{equation}
       \norm{ \gi-H(\xh_i) (\xh_i-x^*) }\leq \frac{L_H}{2}\norm{ \xh_i-x^*}^2.\label{gi-Hi} 
    \end{equation}
    Moreover, by~\zcref{eq:d},~\zcref{E^+}, and~\zcref{der1.4}, we have
      \begin{equation} \label{norm-errsum}
        \begin{aligned}
      \norm{ \errsum} &\leq \ei\big(\norm{ \xh_i - x^*} + \norm{ x^*} +\norm{ \dhat_i}\big)\\
    &\leq  \ei\big(\norm{ \xh_i - x^*} + \norm{ x^*} +   \norm{ (\Hi+\errh)^{-1} }(\norm{ \gi } + \norm{ \errg})\big) \\
    & \leq  \ei\big(\norm{ \xh_i - x^*} + \norm{ x^*} +   \frac{\norm{ \Hi^{-1}} }{1-\nu_i}
    (\norm{ \gi } + \norm{ \errg})\big).  \end{aligned}
    \end{equation}
    We then use \zcref{gi-Hi} to bound the norm of $\gi$:
    \begin{equation} \label{bound_g}
        \begin{aligned}
            \norm{ \gi} &\leq \norm{ \gi - \Hi(\xh_i-x^*)} + \norm{ \Hi(\xh_i-x^*)} 
            \\
            &\leq \frac{L_H}{2}\norm{ \xh_i - x^*}^2 + \norm{ \Hi} \norm{ \xh_i -x^*}. 
        \end{aligned}
    \end{equation}
    Finally, collecting inequalities from~\zcref{firstterm} to~\zcref{bound_g}, and \zcref{E^g} to bound each of the terms in \zcref{bigdis} yields
    \begin{equation}\label{eq:contraction}
      \norm{ \xh_{i+1} - x^*} \leq \alpha_i\norm{\xh_i - x^*}^2 + \beta_i \norm{ \xh_i -x^* } + \limacc,
    \end{equation}
    where $\alpha_i$, $\beta_i$, and $\gamma_i$ are defined
    in \zcref{def:alpha_i_thm}, \zcref{def:beta_i_thm}, and \zcref{def:limacc_thm}.
    We have thus proved that if $\xh_i \in B_{\rho}(x^*)$ and $\nu_i < 1$, then \zcref{eq:contraction} holds.

    Assume now that $\xh_0\in B_{\rho}(x^*)$.
    Let $i_1$ be the first $i$ for which 
    the relative error does not decrease, that is, 
    $\norm{\xh_{i+1} - x^*} \ge \norm{\xh_i - x^*}$.
    Then, for all $i < i_1$, 
    since $\norm{\xh_{i+1} - x^*} < \norm{\xh_i - x^*}$,
    $\xh_{i+1} \in B_{\rho}(x^*)$ remains in the ball. If, moreover,
    $\nu_i<1$ for all $i\le i_1$, then we
    can inductively apply \zcref{eq:contraction} until $i=i_1$.
    Let us assume that there exists a $\theta_{\max} \in [0, 1)$ such that 
    $\theta_i := \alpha_i\norm{\xh_i - x^*} + \beta_i < \theta_{\max}$ for all $i < i_1$.
    Then, if 
    $\limacc \le (1-\theta_{\max})\norm{\xh_i - x^*}$, we have
    \[
      \norm{\xh_{i+1} - x^*} \leq \errconvrate\norm{\xh_i - x^*} + \limacc \leq (\errconvrate + 1 - \theta_{\max})\norm{\xh_i - x^*} <  \norm{\xh_i - x^*}
    \]
    which shows that at $i=i_1$, for the error to stop decreasing, we must necessarily have $\limacc > (1-\theta_{\max})\norm{\xh_i - x^*}$.
    Therefore \zcref{eq:contraction} holds for all $i$ until the first $i_0 \le i_1$
    for which $\norm{\xh_i-x^*} < \limacc/(1-\theta_{\max})$.
\end{proof}
\subsection{Proof of~\zcref{th:grad_norm}} \label{sec:proof_grad_norm}

\begin{proof}[Proof of \zcref{th:grad_norm}]
    If~\zcref{def:nu_i_thm} holds at iteration $i$, by~\zcref{th:rel_err}
$\xh_{i+1}$ is well defined and the error on the iterates is bounded as
in~\zcref{thes2.1}. We want to relate the norm of the gradient at iteration
$i+1$ with that at the previous iteration. 
    
   To do that, let us define $$\omega_i  = \ginext - \gi - \Hi(\xh_{i+1} - \xh_i).$$
    Note that by \zcref{eq:single_eq_newton} and \zcref{eq:d}, we have 
    \begin{align*}
        \ginext = \gi + \Hi(\dhat_i + \errsum) + \omega_i= -\errg - \errh \dhat_i + \Hi \errsum + \omega_i,
    \end{align*}
    which yields, by~\zcref{E^g},~\zcref{E^H} and~\zcref{E^+}, that
    \begin{align}
        \norm{ \ginext } &\leq \norm{ \errg} + \norm{ \errh } \norm{ \dhat_i } + \norm{ \Hi } \norm{ \errsum} + \norm{ \omega_i }\nonumber \\
        & \leq \eg
            + \norm{ \Hi }\norm{ \dhat_i }(\eh + \ei) + \ei \norm{ \Hi } \norm{ \xh_i} + \norm{ \omega_i }. \label{der3.2}
    \end{align}
    By~\zcref{def:err_model_newton}, using~\zcref{lemma22} with $M\equiv \Hi^{-1}\errh$ and \zcref{der0}, \zcref{inv}, we have
    \begin{align}
        \norm{ \dhat_i } &\leq \norm{( \Hi + \errh)^{-1}}(\norm{ \gi } + \norm{ \errg}) \nonumber\\
        &\leq \norm{(I+\Hi^{-1}\errh)^{-1}}\norm{\Hi^{-1}}(\norm{ \gi } + \norm{ \errg}) \nonumber \\
        &\leq \frac{\norm{ \Hi^{-1}}}{1-\nu_i} (\norm{ \gi } + \eg),  \label{der3.4}
    \end{align}
    which gives
    \begin{align}\label{der3.5}
      \norm{ \Hi}\norm{ \dhat_i } (\eh + \ei) \leq \Bigl(\frac{\nu_i}{1-\nu_i} + \frac{\ei \kappa(\Hi)}{1-\nu_i}\Bigr)(\norm{ \gi } + \eg).
   \end{align}
    By \zcref{lemma21} we have
    $\norm{ \omega_i } \leq \frac{L_H}{2}\norm{ \xh_{i+1} - \xh_i}^2.$
    By~\zcref{eq:single_eq_newton} 
    and using \zcref{E^+} and \zcref{der3.4}, it follows that
    \begin{align}   
            \norm{ \xh_{i+1} - \xh_i } &\leq (1+\ei)\norm{ \dhat_i } + \ei \norm{ \xh_i}\nonumber \\
       & \leq \norm{ \Hi^{-1}} \frac{1+\ei}{1-\nu_i}\left(\norm{ \gi} + \eg\right) + \ei \norm{ \xh_i}. \label{der3.7}
        \end{align}
    By the triangle inequality and \zcref{th:rel_err} we also have
    \begin{equation} \label{der3.8}
        \norm{ \xh_{i+1} - \xh_i } \leq (\errconvrate+1)\norm{ \xh_i-x^*} + \limacc,
    \end{equation}
    with $\errconvrate$ defined in \zcref{def:theta_i_thm}.
    By multiplying  \zcref{der3.7} and \zcref{der3.8}  term by term, we obtain:
    \begin{equation} \label{der3.9}
        \begin{aligned}
            \norm{ \omega_i } \leq & \frac{(1+\ei)(\errconvrate+1)}{2(1-\nu_i)}L_H\norm{ \Hi^{-1}} \norm{ \xh_i-x^*} \norm{ \gi } 
            \\
            &+ \frac{1+\ei}{2(1-\nu_i)}L_H\norm{ \Hi^{-1}} \limacc \norm{ \gi} \\
            &+\frac{(1+\ei)(\errconvrate+1)}{2(1-\nu_i)}L_H\norm{ \Hi^{-1}} \norm{ \xh_i-x^*} \eg 
            \\
            & + \frac{1+\ei}{2(1-\nu_i)}L_H\norm{ \Hi^{-1}} \limacc \eg 
            \\
            & + \frac{\errconvrate+1}{2}L_H\norm{ \Hi^{-1}} \norm{ \xh_i - x^*} \ei \norm{ \Hi} \norm{ \xh_i } 
            \\
            &+ \frac{1}{2}L_H \limacc \norm{ \Hi^{-1}} \ei \norm{ \Hi } \norm{ \xh_i },
        \end{aligned}
    \end{equation}
    where the penultimate and the last terms on the right-hand side of the inequality are obtained using $\norm{ \Hi } \norm{ \Hi^{-1} } = \kappa(\Hi) \geq 1$.
    Substituting \zcref{der3.5} and \zcref{der3.9} into \zcref{der3.2} we have
    \begin{align*}
        \norm{ \ginext} \leq \phi_i\norm{ \gi } + \limg,
    \end{align*}
    with
    \begin{equation} \label{eq:def_phi_i}
        \begin{aligned}
            \phi_i := &\frac{1+\ei}{2(1-\nu_i)}\left((\errconvrate + 1)\mu_i + \tau_i\right)
            \\
            &+ \frac{\nu_i + \ei \kappa(\Hi)}{1-\nu_i}, 
        \end{aligned}   
    \end{equation}
    and 
    \begin{equation} \label{eq:def_limg}
        \begin{aligned}
            \limg := &\eg\left(1 + \frac{\nu_i + \ei \kappa(\Hi)}{1-\nu_i}\right)
            \\
            &+ \eg\left(\frac{(1+\ei)}{2(1-\nu_i)}(\errconvrate + 1)\mu_i + \tau_i\right)
            \\
            &+ \frac{1}{2}\ei \norm{\Hi}\norm{\xh_i}\left((\errconvrate + 1)\mu_i + \tau_i\right) 
            \\
            &+ \ei \norm{ \Hi} \norm{ \xh_i },
        \end{aligned}
    \end{equation}
    where
    $\mu_i = L_H\norm{ \Hi^{-1} }\norm{ \xh_i - x^* }$ and $\tau_i = L_H\limacc\norm{ \Hi^{-1} }$, with $\limacc$ defined in \zcref{th:rel_err}. Rearranging the terms in \zcref{eq:def_phi_i} and \zcref{eq:def_limg}, we prove the first result of the theorem. 

    For the second part of the theorem, we proceed as for the proof
    of~\zcref{th:rel_err} in~\zcref{sec:proof_rel_err}, to prove that the gradient
    norm decreases linearly at a rate at least $\phi_i$ until it reaches $\limg/(1-\phi_{\max})$.
\end{proof}
\end{document}